\documentclass[a4paper,11pt,onecolumn]{amsart}
\usepackage{amsfonts}

\usepackage{amscd}
\usepackage{amsmath}
\usepackage{amssymb}
\usepackage{amsthm}
\usepackage{epsf}
\usepackage{latexsym}
\usepackage{verbatim}
\usepackage{color}
\usepackage{todonotes}
\usepackage{enumitem}
\input epsf.tex
\usepackage{bbm}
\usepackage{mathtools}

\usepackage{epstopdf}
\usepackage{graphicx}

\usepackage{centernot}
\usepackage{mathtools}
\usepackage{ stmaryrd }
\usepackage[all]{xy}

\usepackage{tikz}
\usetikzlibrary{shapes.geometric, arrows}

\usepackage{marginnote}

\theoremstyle{plain}
\newtheorem{theorem}{Theorem}
\newtheorem{definition}{Definition}[section]
\newtheorem{lemma}[definition]{Lemma}
\newtheorem{corollary}[definition]{Corollary}
\newtheorem{proposition}[definition]{Proposition}
\theoremstyle{definition}

\newtheorem*{theoremNotPLMorse}{Theorem~\ref{thm:NotPLMorse}}

\newtheorem*{retractiontheorem}{Theorem~\ref{t:retraction}}
\newtheorem*{homotopyequivtheorem}{Theorem~\ref{t:homotopyequiv}}

\theoremstyle{remark}
\newtheorem{remark}[definition]{Remark}

\newcommand{\R}{\textrm{R}}
\newcommand{\CA}{\mathcal{C}(\mathcal{A})}

\newcommand{\MflatF}{M_{\tiny \mbox{flat}}(F)}

\newcommand*\notocchapter[1]{%
  \if@openright\cleardoublepage\else\clearpage\fi
  \thispagestyle{empty}\global\@topnum\z@
  \@afterindenttrue
  \let\@secnumber\@empty
  \@makeschapterhead{#1}\@afterheading
}

\makeatother

\begin{document} 
\bibliographystyle{plain}
\title[Topological Complexity of ReLU networks]{Local and global topological complexity measures of ReLU neural network functions}

\author{J. Elisenda Grigsby}
\thanks{JEG was partially supported by Simons Collaboration grant 635578 and NSF grant  DMS - 2133822.}
\address{Boston College; Department of Mathematics; 522 Maloney Hall; Chestnut Hill, MA 02467}
\email{grigsbyj@bc.edu}

\author{Kathryn Lindsey}
\thanks{KL was partially supported by NSF grants DMS - 1901247 and DMS - 2133822.}
\address{Boston College; Department of Mathematics; 567 Maloney Hall; Chestnut Hill, MA 02467}
\email{lindseka@bc.edu}

\author{Marissa Masden}
\address{University of Oregon; Department of Mathematics, Fenton Hall; Eugene, OR 97403-1222 USA}
\email{mmasden@uoregon.edu}

\begin{abstract} 
We apply a generalized piecewise-linear (PL) version of Morse theory due to Grunert-K{\"u}hnel-Rote to define and study new local and global notions of topological complexity for fully-connected feedforward ReLU neural network functions $F: \mathbb{R}^n \rightarrow \mathbb{R}$.  Along the way, we show how to construct, for each such $F$, a canonical polytopal complex $\mathcal{K}(F)$ and a deformation retract $\mathbb{R}^n \rightarrow \mathcal{K}(F)$, yielding a convenient compact model for performing calculations. We also give a  construction showing that local complexity can be arbitrarily high.
\end{abstract}

\maketitle

\section{Introduction}
It is well-known (e.g. \cite{AroraBasu}) that the class of functions $F: \mathbb{R}^n \rightarrow \mathbb{R}$ realizable by feedforward ReLU neural networks is precisely the class of finite piecewise linear (PL) functions (Definition \ref{defn:finitePL}). Less well-understood is how architecture impacts function complexity and how complexity evolves during training. That is, for a fixed architecture of fully-connected, feedforward ReLU neural network, how does the set of finite PL functions realizable by that architecture sit inside the set of all finite PL functions? Do the dynamics of stochastic gradient descent preference certain classes of finite PL functions over others? In order to tackle these questions in a systematic way, we need a good framework for characterizing the complexity of finite PL functions. In the present work, we draw on algebraic topology and a PL version of Morse theory to develop a new notion of complexity.

The idea of using Betti numbers of sublevel sets as a measure of complexity for neural network functions was initiated in \cite{BianchiniScarselli}.  Recall that the Betti numbers of a topological space are the ranks of its {\em homology groups} (see e.g. \cite{Hatcher}), and they are invariants of the space up to homeomorphism (in fact, up to the weaker notion of homotopy equivalence, see Definition \ref{defn:HE}).  In \cite{BianchiniScarselli} and, later, \cite{GussSalakhutdinov}, the  {\em homological} or  {\em topological} complexity of a continuous function $F: \mathbb{R}^n \rightarrow \mathbb{R}$  was defined as the collection of Betti numbers of the (single) sublevel set $F_{\leq 0} := F^{-1}(-\infty, 0].$  In \cite{BianchiniScarselli}, the authors explore how homological complexity grows with respect to depth and width for neural networks with polynomial and $\arctan$ activation functions, and in \cite{GussSalakhutdinov}, the authors use {\em persistent homology}, introduced by Zomorodian-Carlsson in \cite{ZomorodianCarlsson}, to perform an empirical study of the homological complexity of ReLU neural networks of varying architectures. We note that modern practitioners favor the ReLU activation function for supervised learning problems involving regression.

Our study begins with the observation that another natural and complementary notion of complexity of a continuous function \[F: \mathbb{R}^n \rightarrow \mathbb{R}\] is the {\em number} of homeomorphism classes of its sublevel sets, \[F_{\leq a} := F^{-1}(-\infty,a], \mbox{ for } a \in \mathbb{R}.\] 
For a \emph{smooth} function $F$, Morse theory provides a toolkit for not only computing algebro-topological invariants of the sublevel sets (e.g., their Betti numbers) but also understanding how they change as the threshold $t = a$ varies. Put simply, for a {\em smooth} function $F: \mathbb{R}^n \rightarrow \mathbb{R}$, the homeomorphism class of a sublevel (or superlevel) set can only change across a {\em critical point}, where the gradient of $F$ vanishes. Moreover, generic smooth functions are {\em Morse}, which means that the critical points are isolated and admit a standard quadratic local model.  These conditions yield 
a straightforward description of how the topology of the sublevel set changes as the threshold crosses the critical value; the change is local, controlled, and relatively easy to describe. 

\color{black}
\begin{figure}
	\includegraphics[width=1.5in]{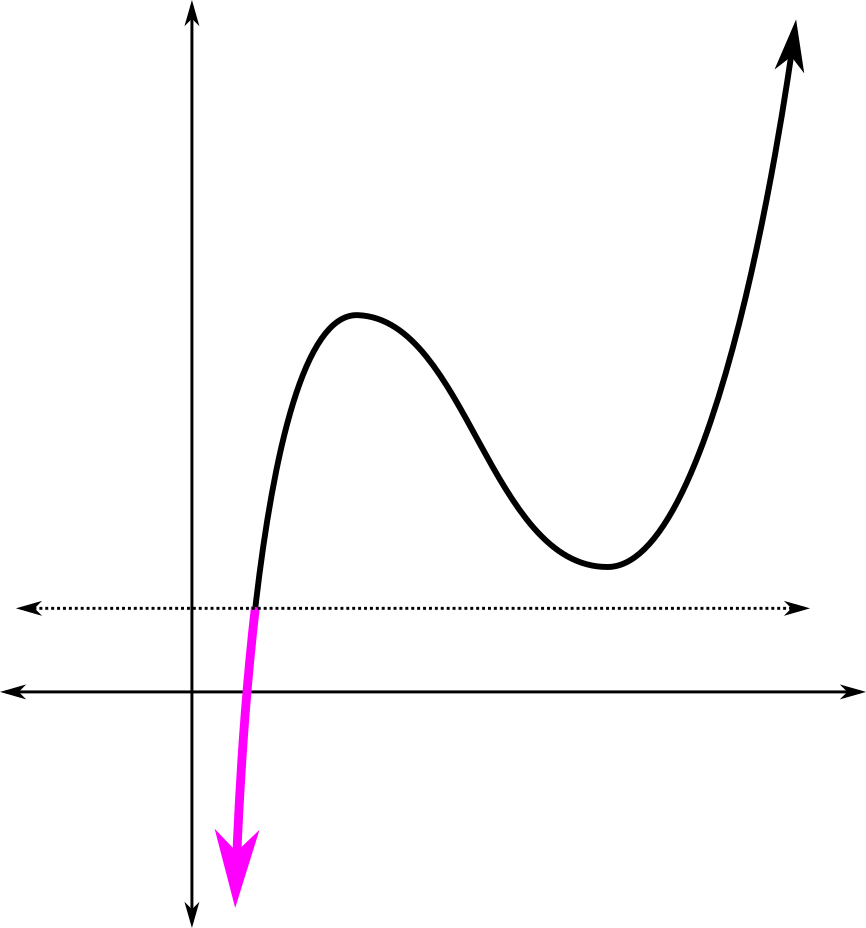}	\includegraphics[width=1.5in]{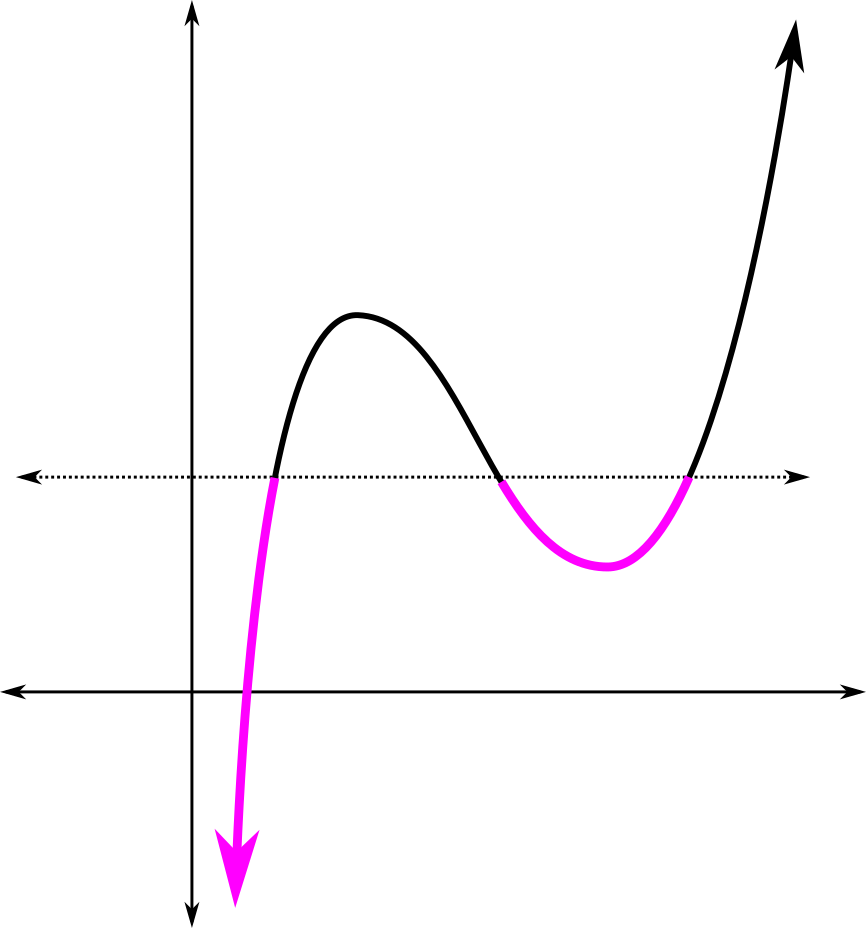}  \includegraphics[width=1.5in]{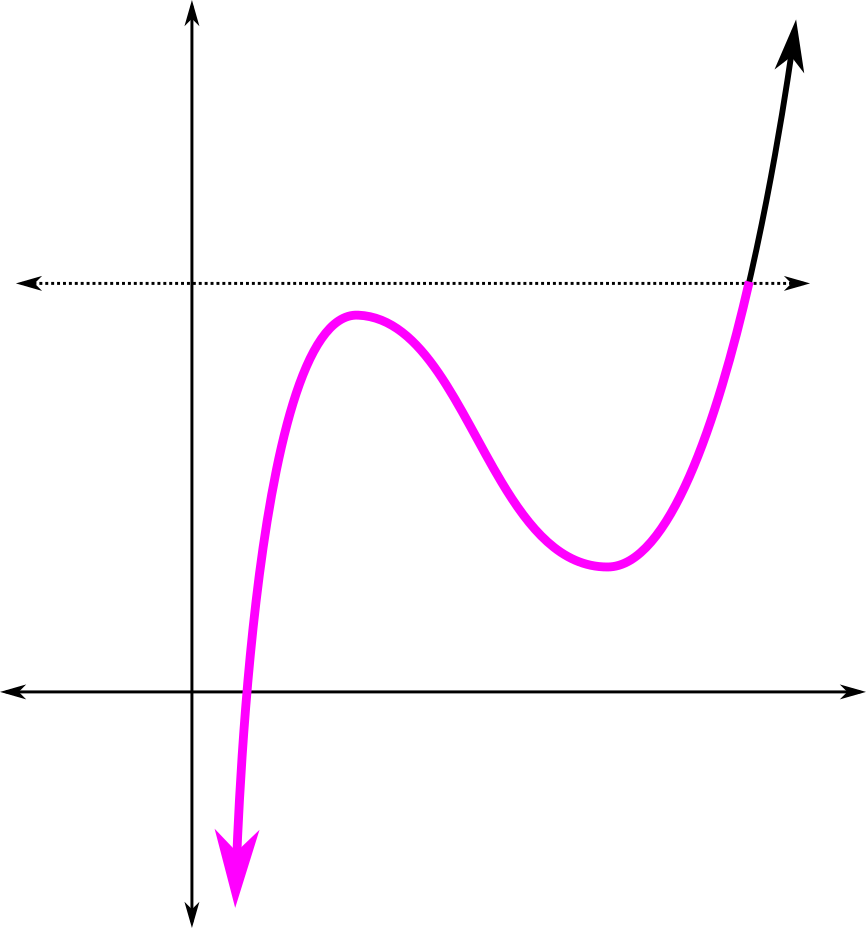}
	\caption{Algebro-topological invariants of the sublevel sets $F_{\leq t}$ of a function $F: \mathbb{R}^n \rightarrow \mathbb{R}$ for varying $t \in \mathbb{R}$ give strong information about the complexity of the function.}
	\label{f:sublevelsets}
\end{figure}

For PL functions $\mathbb{R}^n \rightarrow \mathbb{R}$, it is still natural to define a notion of complexity based on algebro-topological descriptions of homeomorphism classes of sublevel (or superlevel) sets. However, in contrast to the smooth setting, a generic finite PL function is not necessarily PL Morse (see Section \ref{sec:PLMorse}); consequently, describing how the homological complexity of a sublevel set $F_{\leq a}$ varies with $a \in \mathbb{R}$ requires a more delicate analysis.

Following ideas dating back to Banchoff \cite{Banchoff}, we focus our attention on the {\em flat} cells of a ReLU neural network map $F: \mathbb{R}^n \rightarrow \mathbb{R}$. The \emph{canonical polyhedral complex} $\mathcal{C}(F)$ associated to a feedforward ReLU neural network map $F:\mathbb{R}^n \to \mathbb{R}$ was introduced in \cite{GL}; informally, it is a decomposition of $\mathbb{R}^n$ into convex polyhedra (cells) of varying dimension such that $F$ is affine-linear on each cell.  

 The flat cells of $F$ are, by definition, the cells of $\mathcal{C}(F)$ on which $F$ is constant. Flat cells in the PL category should be viewed as the appropriate analogues of critical points in the smooth category, with the caveat that not every flat cell is critical. Explicitly, we prove:

\begin{theorem} \label{thm:HomEquivTransversal} Let $F: \mathbb{R}^n \rightarrow \mathbb{R}$ be a finite PL function, and let $a \leq b \in \mathbb{R}$ be thresholds. If the polyhedral complex $\mathcal{C}(F)_{F_{\in [a,b]}}$ contains no flat cells, then the sublevel sets $F_{\leq a}$ and $F_{\leq b}$ are homotopy equivalent, as are the superlevel sets $F_{\geq a}$ and $F_{\geq b}$.
\end{theorem}

\noindent Here, $\mathcal{C}(F)_{F_{\in [a,b]}}$ is a level set complex (see \S \ref{s:polybackground} for the precise definition); it is, roughly speaking, the restriction of the complex $\mathcal{C}(F)$ to the set $F^{-1}([a,b])$. 
Homotopy equivalence is an equivalence relation on topological spaces that is weaker than homeomorphism but strong enough to imply the equality of all standard invariants coming from algebraic topology (cf. Section \ref{sec:Hcompxty} and \cite{Hatcher}). Moreover, since a finite PL map has finitely many flat cells, an immediate consequence of Theorem \ref{thm:HomEquivTransversal} is that there are finitely many homotopy equivalence classes of sublevel sets -- i.e. for a feedforward ReLU neural network function $F$, we need only consider finitely many sublevel sets $F_{\leq a}$, $a \in \mathbb{R}$,  to capture all of the topological complexity information associated to $F$. 

To this end, we may first consider the sublevel sets $F_{\leq a}$ (resp. superlevel sets $F_{\geq a}$) for very negative (resp. very positive) values of $a \in \mathbb{R}$.  
Indeed, if the images of all flat cells are contained in the compact interval $[-M,M]$, then the homology of $F_{\leq a}$ and $F_{\leq b}$ agree for all $a, b < -M$ and the homology of $F_{\geq a}$ and $F_{\geq b}$ agree for all $a,b > M$. Consequently, a first (coarse) description of the topological complexity of $F$ is given by:

\begin{itemize}
	\item the collection of Betti numbers of the set $F_{\leq a}$ for any $a < -M$ and 
	\item the collection of Betti numbers of $F_{\geq a}$ for any $a > M$.
\end{itemize} 

Next, we may investigate how the sublevel sets $F_{\leq t}$ (and superlevel sets $F_{\geq t}$) change as we vary $t \in [-M,M]$.  In analogy to the smooth setting, our preference would be to apply PL Morse theory (as developed in \cite{Grunert}) to the class of finite PL maps realized by ReLU neural network functions, but we quickly see that ReLU neural network functions have a nonzero probability of being non PL Morse (Definition \ref{defn:PLMorse}):

\begin{theorem} \label{thm:NotPLMorse} Let $F: \mathbb{R}^n \rightarrow \mathbb{R}$ be a ReLU neural network map with hidden layers of dimensions $(n_1, \ldots, n_\ell)$. The probability that $F$ is  PL Morse is 
\begin{equation} \label{eq:nonPLMorseProb}
\left\{\begin{array}{cl} = \frac{\sum_{k=n+1}^{n_1} {n_1 \choose k}}{2^{n_1}} & \mbox{if $\ell = 1$}\\
		 \leq \frac{\sum_{k=n+1}^{n_1} {n_1 \choose k}}{2^{n_1}}  & \mbox{if $\ell > 1.$}\end{array}\right.
		 \end{equation}
\end{theorem}

We remark that the expression \eqref{eq:nonPLMorseProb} above for the probability is \[\left\{\begin{array}{cl} 0 & \mbox{if $n_1 \leq n$}\\
																			\leq \frac{1}{2} & \mbox{if $n_1 \leq 2n+1$}.\end{array}\right.\]
Additionally, note that for architectures whose hidden layers all have the same dimension (a common choice), the upper bound given by  Theorem~\ref{thm:NotPLMorse} is significantly lower for deep networks than for shallow networks of the same total dimension. For ReLU neural networks with $>1$ hidden layer, we expect the true probability that $F$ is PL Morse to be much lower than the given upper bound. 

Thus, in order to measure {\em how} the sublevel sets $F_{\leq t}$ change as we vary the threshold $t$, we need to do more work. Following \cite{GrunertRote}, we associate to each (connected component of)\footnote{Flat cells of PL maps need not be isolated.} flat cell(s) $K$ at level $t$ its {\em local homological complexity} (or {\em local $H$--complexity} for short), which we define to be the rank of the relative homology of the pair $(F_{\leq t}, F_{\leq t} \setminus K)$. (Readers unfamiliar with relative homology may informally view the local $H$--complexity of $K$ as the (total rank of the) homology of the quotient space obtained by collapsing everything in $F_{\leq t}$ outside of a neighborhood of $K$ to a point; for more detail see \cite{Hatcher}.) We may then define the {\em total $H$--complexity} of a finite PL map to be the sum of all the local $H$--complexities. It is noted in \cite{GrunertRote} that PL Morse {\em regular} points have local $H$--complexity $0$ and PL Morse {\em nondegenerate-critical} points have local $H$--complexity $1$.

To prove Theorem \ref{thm:HomEquivTransversal} we draw heavily upon Grunert's work in \cite{Grunert}, where he proves an analogous result for \emph{bounded} polyhedral complexes, aka {\em polytopal complexes}.

 We apply results in convex geometry and combinatorics (cf. \cite{Schrijver}) to construct, 
  for every ReLU neural network map $F: \mathbb{R}^n \rightarrow \mathbb{R}$ with canonical polyhedral complex $\mathcal{C}(F)$, a \emph{canonical polytopal complex}, $\mathcal{K}(F)$, and a homotopy equivalence $|\mathcal{C}(F)| \sim |\mathcal{K}(F)|$ (we use $| \cdot |$ to denote the underlying set of a polyhedral complex) that respects the finite PL map $F$:

\begin{theorem} \label{t:retraction}
Let $\mathcal{C}$ be a polyhedral complex in $\mathbb{R}^n$, let $F:|\mathcal{C}| \to \mathbb{R}$ be linear on cells of $\mathcal{C}$, and assume $|\mathcal{C}_{F \in [a,b]}| \subseteq |\mathcal{C}|$ contains no flat cells.  Then there exists a polytopal complex $\mathcal{D}$ with $|\mathcal{D}| \subset |\mathcal{C}|$ and a strong deformation retraction $\Phi:|\mathcal{C}| \to |\mathcal{D}|$ (which induces a cellular map $\mathcal{C} \to \mathcal{D}$) such that
$\Phi$ is $F$-level preserving on $|\mathcal{C}_{F \in [a,b]}|$ and for every $c \in [a,b]$,
\begin{equation} \label{eq:PhiProperties}
\Phi( |\mathcal{C}_{F \leq c}| ) = |\mathcal{D}_{F \leq c}|, \quad \Phi( |\mathcal{C}_{F \geq c}| ) = |\mathcal{D}_{F \geq c}|.
\end{equation}
Furthermore, there is an $F$-level preserving PL isotopy  
$$\phi: |\mathcal{D}|_{F = c} \times [a,b] \to |\mathcal{D}_{F \in [a,b]}|.$$
\end{theorem}

Combined with results of Grunert (\cite{Grunert}) for polytopal complexes, this immediately implies:

\begin{theorem} \label{t:homotopyequiv}
Let $\mathcal{C}$ be a polyhedral complex in $\mathbb{R}^n$, let $F:|\mathcal{C}| \to \mathbb{R}$ be linear on cells of $\mathcal{C}$, and assume $|\mathcal{C}_{F \in [a,b]}| \subseteq |\mathcal{C}|$ contain no flat cells.  Then for any threshold $c \in [a,b]$, $|\mathcal{C}_{F \leq c}|$ is homotopy equivalent to $|\mathcal{C}_{F \leq a}|$; also $|\mathcal{C}_{F \geq c}|$ is homotopy equivalent to $|\mathcal{C}_{F \geq b}|$. \end{theorem}

Because polytopal complexes admit finite simplicial subdivisions and there exist finite time algorithms for computing homology for finite simplicial complexes,  Theorems \ref{t:retraction} and \ref{t:homotopyequiv} together have the following important practical consequence, which we state informally:

\vskip 10pt
{\em For every flavor of $H$--complexity described in this paper and every finite PL map $F: \mathbb{R}^n \rightarrow \mathbb{R}$, there exists a finite time algorithm for computing the $H$--complexity of $F$.}
\vskip 10pt

After establishing these foundational preliminaries, we present, in Section \ref{sec:LocCmpxArbLarge}, a construction proving that local $H$--complexity of a ReLU neural network function can be arbitrarily large. 
\tableofcontents

\subsection*{Acknowledgments}
The authors would like to thank B. Hanin, M. Telgarsky, and D. Rolnick for interesting conversations.
\color{black}

\color{black}
\section{Preliminaries}  \label{s:polybackground}

We briefly recall some definitions and constructions in convex geometry, along with how they appear in the geometry of ReLU neural networks. We refer the reader to \cite{Grunbaum, Grunert, GL} for more details.

\subsection{Polyhedral complexes}

An (affine) hyperplane $H \subseteq \mathbb{R}^n$ is any subset of the form $H = \{\vec{x} \in \mathbb{R}^n \,\,|\,\,\vec{w} \cdot \vec{x} + b = 0\}$ for some $b \in \mathbb{R}$ and nonzero $\vec{w} \in \mathbb{R}^n$. If $H$ has been specified as above by a pair $(\vec{w},b)$, then it is equipped with a co-orientation in the direction of $\vec{w}$, dividing $\mathbb{R}^n$ into two closed half-spaces:
\begin{eqnarray*} \label{eq:coorient}
	H^+ &:=& \{\vec{x} \in \mathbb{R}^n \,\,|\,\, \vec{w} \cdot \vec{x} + b \geq 0\}\\
	H^- &:=& \{\vec{x} \in \mathbb{R}^n \,\,|\,\,  \vec{w} \cdot \vec{x} + b \leq 0\}, 
\end{eqnarray*}
A {\em polyhedral set $\mathcal{P}$ in $\mathbb{R}^n$} is an intersection of finitely many closed half-spaces in $\mathbb{R}^n.$  That is, a polyhedral set is any subset of $\mathbb{R}^n$ that can be written as $H_1^+ \cap \ldots \cap H_k^+$ for some co-oriented hyperplanes $H_1, \ldots, H_k$. A {\em polytope} in $\mathbb{R}^n$ is a bounded polyhedral set. 

Given a polyhedral set $\mathcal{P}$ and hyperplane $H$ in $\mathbb{R}^n$, $H$ is called a {\em cutting hyperplane} of  and is said to \emph{cut} $P$ if $H$ has nonempty intersection with the interior of $P$, and $H$ is called a  {\em supporting hyperplane} of $P$ and is said to {\em support} $P$ if $H$ does not cut $P$ and $H \cap P \neq \emptyset$.   The \emph{dimension} of a polyhedral set  $P$ is the dimension of its \emph{affine hull} $\textrm{aff}(P)$, which is  the  the intersection of all affine-linear subspaces of $\mathbb{R}^n$ that contain $P$. A subset $F$ of a polyhedral set $P$ is said to be a {\em face} of $P$ if either $F = \emptyset$, $F =P $, or $F = H \cap P$ for some supporting hyperplane of $\mathcal{P}$.  

A {\em polyhedral complex} $\mathcal{C}$ of dimension $d$ is a {\em finite}  set 
 of polyhedral sets of dimension $k$, for $0 \leq k \leq d$, called the \emph{cells} of $\mathcal{C}$, such that i)  If $P \in \mathcal{C}$, then every face of $P$ is in $\mathcal{C}$, and ii)  if $P, Q \in \mathcal{C}$, then $P \cap Q$ is a single mutual face of $P$ and $Q$.  A \emph{polytopal complex} is a polyhedral complex in which all the polyhedral cells are polytopes (i.e. bounded). A \emph{simplicial complex} is a polytopal complex in which all polytopes are simplices. Any polytopal complex admits a finite simplicial subdivision, which can, if desired, be constructed without introducing new vertices, cf. \cite[Fact 1.25]{Grunert}. It is immediate that a map that is affine-linear on cells of a simplicial complex is uniquely determined by the images of the vertices ($0$--cells).  The \emph{domain} or \emph{underlying set} $|\mathcal{C}|$ of a polyhedral complex $\mathcal{C}$ is the union of its cells. If $\mathcal{C}$ is a polyhedral complex embedded in $\mathbb{R}^n$ and $|\mathcal{C} | = \mathbb{R}^n$, we call $\mathcal{C}$ a {\em polyhedral decomposition of $\mathbb{R}^n$}. 

The \emph{intersection complex} (c.f. \cite{Grunert} Definition 1.16) $\mathcal{K} \cap \mathcal{L}$ of two polyhedral complexes $\mathcal{K}$ and $\mathcal{L}$ in $\mathbb{R}^n$ is the complex consisting of all pairwise intersections $S \cap T$  of cells $S \in \mathcal{K}$ and $T \in \mathcal{L}$, i.e. $$\mathcal{K} \cap \mathcal{L} \coloneqq \{S \cap T \mid S \in \mathcal{K}, T \in \mathcal{L}\}.$$
By construction,  $|\mathcal{K} \cap \mathcal{L}| = |\mathcal{K}| \cap |\mathcal{L}|$.  Now allowing $\mathcal{K}$ and $\mathcal{L}$ to be embedded in different Euclidean spaces, given a function $f:|\mathcal{K}| \to |\mathcal{L}|$ that is affine-linear on cells of $\mathcal{K}$, the \emph{level set complex} (c.f. \cite{Grunert} Definition 2.6) of $\mathcal{K}$ with range $\mathcal{L}$ is the polyhedral complex 
$$\mathcal{K}_{f \in \mathcal{L}}\coloneqq \{K \cap f^{-1}(L) \mid K \in \mathcal{K}, L \in \mathcal{L}\}.$$ When the map $f$ is clear from context, we will suppress the $f$ and write $\mathcal{K}_{\in \mathcal{L}}$.
By construction, $|\mathcal{K}_{\in \mathcal{L}}| = |\mathcal{K}| \cap f^{-1}(|\mathcal{L}|)$.  In particular, when $\mathcal{L}$ is a interval complex in $\mathbb{R}$ of the form $\{t\}, [t_1,t_2], (-\infty,t]$ or $[t,\infty)$ (together with their faces), we denote the associated level set complexes $\mathcal{K}_{=t}$, $\mathcal{K}_{[t_1,t_2]}$, $\mathcal{K}_{\leq t}$ and $\mathcal{K}_{\geq t}$, respectively. We call $\mathcal{K}_{\leq t}$ (resp. $\mathcal{K}_{\geq t}$) a \emph{sublevel} (resp. \emph{superlevel}) \emph{set complex}. 

\subsection{Feedforward ReLU neural networks: genericity, transversality and canonical polyhedral complex} \label{ss:ReLUNNs}

The ReLU activation function, $\textrm{ReLU}:\mathbb{R} \to \mathbb{R}$, is defined by $\textrm{ReLU}(x) = \max\{0,x\}$; for any $n \in \mathbb{N}$, denote by $\sigma$ the function $\mathbb{R}^n \to \mathbb{R}^n$ that applies $\textrm{ReLU}$ in each coordinate.  
 A fully connected, feedforward ReLU neural network defined on $\mathbb{R}^{n_0}$, $n_0 \in \mathbb{N}$, with one-dimensional output is a finite sequence of natural numbers $n_0,n_1,\dots,n_m$ together with affine-linear maps $A_i:\mathbb{R}^{n_i} \to \mathbb{R}^{n_{i+1}}$ for $i = 0,\dots,m$.  This determines a function 
   \[F:\mathbb{R}^{n_0}  \xrightarrow{F_1 = \sigma \circ A_1} \mathbb{R}^{n_1}  \xrightarrow{F_2 = \sigma \circ A_2} \ldots \xrightarrow{F_m = \sigma \circ A_{m}} \mathbb{R}^{n_m}  \xrightarrow{G = A_{m+1}} \mathbb{R}^1.\]
   Such a neural network is said to  be of
 \emph{architecture} $(n_0,\dots,n_m; 1)$, 
 \emph{depth}  $m+1$, and \emph{width} 
 $\max \{n_1,\ldots,n_m,1\}$.   
  The $k^{\textrm{th}}$ \emph{layer map} of such a neural network is the composition $\sigma \circ A_k$ for $k = 1,\dots,m$ and is the map $G = A_k$ for $k=m+1$.    For each $i \in \{1, \ldots, m\}$ and $j \in \{1, \ldots, n_{i}\}$, the {\em node map}, $F_{i,j}$, is the map  \[\pi_j \circ A_i \circ F_{i-1} \circ \ldots \circ F_1 : \mathbb{R}^{n_0} \rightarrow \mathbb{R},\] where $\pi_j$ denotes projection onto the $j$th coordinate.
 
The standard polyhedral decomposition $\mathcal{C}^{\textrm{std}}(\mathbb{R}^n)$ of $\mathbb{R}^n$ is the polyhedral decomposition induced by the standard coordinate hyperplane arrangement  $\mathcal{A}^{\textrm{std}}$ in $\mathbb{R}^n$ (consisting of the $n$ hyperplanes defined by setting one coordinate to be $0$).  Let \[F: \mathbb{R}^{n_0}  \xrightarrow{F_1} \mathbb{R}^{n_1}  \xrightarrow{F_2} \ldots \xrightarrow{F_m } \mathbb{R}^{n_m}  \xrightarrow{A_{m+1}} \mathbb{R}^1\] be a feedforward ReLU neural network (all layer maps except the last map $A_{m+1}$ have a ReLU factor). 
The reader can easily verify that the \emph{canonical polyhedral complex} $\mathcal{C}(F)$ defined in \cite{GL} has the following description as an intersection complex:
$$\mathcal{C}(F) \coloneqq
\mathcal{I}_{F_1 \in \mathcal{C}^{std}(\mathbb{R}^{n_1})} \cap \mathcal{I}_{F_2 \circ F_1 \in \mathcal{C}^{std}(\mathbb{R}^{n_2})} \cap \ldots \cap \mathcal{I}_{F_m \circ \ldots \circ F_1 \in \mathcal{C}^{std}(\mathbb{R}^{n_m})} $$
where $\mathcal{I}$ is the polyhedral decomposition of $\mathbb{R}^{n_0}$ consisting of a single $n_0$-cell.  By construction, $F$ is affine-linear on cells of $\mathcal{C}(F)$.

The \emph{affine solution set arrangement} associated to a layer map $\sigma \circ A_{i+1} :\mathbb{R}^{n_i} \to \mathbb{R}^{n_{i+1}}$  is the the collection $S$ of sets $S_j \coloneqq \{\vec{x} \in \mathbb{R}^{n_i} :  \pi_j(A(\vec{x}))= 0\}$, for $1 \leq j \leq n_{i+1}$, where $\pi_j$ denotes projection onto the $j$th coordinate.  We call such an affine solution set \emph{generic} if for all subsets $\{S_{i_1} , \ldots , S_{i_p}\} \subseteq S$, it is the case that $S_{i_1} \cap \ldots \cap S_{i_p}$ is an affine linear subspace of $\mathbb{R}^i$ of dimension $i - p$, where a negative-dimensional intersection is understood to be empty. In particular, each element of a generic affine solution set arrangement is a hyperplane. We call the neural network map $F$ \emph{generic} if each of its layer maps has a generic affine solution set arrangement (\cite{GL}).

The $1$-skeleton, $\mathcal{C}(F)_1$, of the canonical polyhedral complex is an embedded linear graph in $\mathbb{R}^{n_0}$.  This graph has a natural partial orientation (see \cite{GL}), which we call the $\nabla F$--orientation: If $F$ is nonconstant on $C$, orient $C$ in the direction in which $F$ increases. If $F$ is constant on $C$, we leave it unlabeled and refer to $C$ as a {\em flat} edge.

We say that a cell $C$ in the canonical polyhedral complex $\mathcal{C}(F)$  is \emph{flat} if $F$ is constant on that cell.  Note that $C \in \mathcal{C}(F)$ is flat if and only if all $1$--dimensional faces of $C$ are unoriented with respect to the $\nabla F$--orientation. \color{black}The following lemma is immediate from the definition of a transversal threshold in \cite{GL}; readers unfamiliar with \cite{GL} may take this lemma as the definition of a nontransversal threshold.

\begin{lemma} \label{lem:flatNT} For a neural network function $F: \mathbb{R}^n \rightarrow \mathbb{R}$, a threshold $t \in \mathbb{R}$ is \emph{nontransversal} if and only if it is the image of a flat cell $C$ in the canonical polyhedral complex $\mathcal{C}(F)$.
\end{lemma}
\color{black}

\noindent A neural network map 
\[F:\mathbb{R}^{n_0}  \xrightarrow{F_1 = \sigma \circ A_1} \mathbb{R}^{n_1}  \xrightarrow{F_2 = \sigma \circ A_2} \ldots \xrightarrow{F_m = \sigma \circ A_{m}} \mathbb{R}^{n_m}  \xrightarrow{G = A_{m+1}} \mathbb{R}^1.\]
is called \emph{transversal} (\cite{GL}) if, for each $i \in \{1,\ldots,m\}$ and each $j \in \{1,\ldots,n_i\}$, $t=0$ is a transversal threshold for the node map $F_{i,j}:\mathbb{R}^{n_0} \to \mathbb{R}$.  

 \color{black}

The following lemma was essentially proved in \cite{GL}:

\begin{lemma} \label{lem:polysubcpx} Let $F: \mathbb{R}^n \rightarrow \mathbb{R}$ be a ReLU neural network function and $a \in \mathbb{R}$ a threshold. Then the level set $F_{=a}$, the sublevel set $F_{\leq a}$, and the superlevel set $F_{\geq a}$ are domains of imbedded polyhedral complexes in $\mathbb{R}^n$. Moreover, if $F$ is generic and transversal and $a \in \mathbb{R}$ is a transversal threshold, then $F_{=a}$ has dimension $n-1$. 
\end{lemma}

\begin{proof} $F$ is linear on cells of the canonical polyhedral complex $\mathcal{C}(F)$, so \cite[Lem. 2.5]{Grunert} tells us that $F_{=a}$ (resp., $F_{\leq a}$ and $F_{\geq a}$) is the domain of the polyhedral subcomplex $\mathcal{C}(F)_{\in \{a\}}$ (resp., $\mathcal{C}(F)_{\in \{(-\infty,a]\}}$ and $\mathcal{C}(F)_{\in \{[a, \infty)\}}$ in the level set complex associated to the polyhedral decomposition \[\{(-\infty,a], \{a\}, [a, \infty)\}\] of the codomain, $\mathbb{R}$.  If $F$ is generic and transversal and $a$ is a transversal threshold, \cite[Lem. 5.6]{GL} tells us $\mathcal{C}(F)_{\in \{a\}}$ has dimension $n-1$, as desired. 
\end{proof}

\section{PL Morse theory and flat versus critical cells} \label{sec:PLMorse}

As indicated in the introduction, studying the topology of sublevel sets of finite PL functions is quite a bit more involved than it is in the smooth case. The purpose of this section is to introduce the notion of a PL Morse function (using the definition from \cite{GrunertRote}, see also \cite{Grunert}) and prove that for each architecture of depth $\geq 2$, a positive measure subset of ReLU neural network functions fail to be PL Morse (Definition \ref{defn:PLMorse}). 

Precisely, we have:

\begin{theoremNotPLMorse} Let $F: \mathbb{R}^n \rightarrow \mathbb{R}$ be a ReLU neural network map with hidden layers of dimensions $(n_1, \ldots, n_\ell)$. The probability that $F$ is  PL Morse is 
\[\left\{\begin{array}{cl} = \frac{\sum_{k=n+1}^{n_1} {n_1 \choose k}}{2^{n_1}} & \mbox{if $\ell = 1$}\\
		 \leq \frac{\sum_{k=n+1}^{n_1} {n_1 \choose k}}{2^{n_1}}  & \mbox{if $\ell > 1$}\end{array}\right.\]
\end{theoremNotPLMorse}

Indeed, we believe that the probability that a given ReLU neural network function is non PL Morse grows with depth. Thus, studying the topology of the sublevel sets of ReLU neural network functions requires establishing more flexible results governing how the sublevel set varies with the threshold.  We will introduce suitable notions for our purposes in Section \ref{sec:Hcompxty}.  

Unsurprisingly, the key to understanding how the topology of sublevel sets changes as one varies the threshold are the PL analogues of points where the gradient of the function vanishes, cf. Theorem \ref{thm:HomEquivTransversal}. These are the so-called {\em flat} or {\em constant} cells (Definition \ref{defn:flatcell}), which map to {\em nontransversal thresholds} (cf. Lemma \ref{lem:flatNT}). Indeed, one should view the analogues of critical cells (both Morse and non-Morse) in the PL category as an appropriate subset of the flat cells. For experts familiar with the ways ReLU neural network functions can degenerate, Lemma \ref{lem:Flatncell} should provide a clear explanation for why most ReLU neural network functions are not PL Morse. 

\subsection{PL $n$--manifolds and PL Morse functions}
We begin with the following preliminaries about PL $n$--manifolds and finite PL maps.

\begin{definition} Let $U \subseteq \mathbb{R}^n, V \subseteq \mathbb{R}^m$ be open sets. A function $f: U \rightarrow V$ is said to be {\em piecewise affine-linear} (PL) if it is continuous and there is a locally-finite decomposition $U = \bigcup_{i \in I} K_i$ into connected, closed subsets $K_i \subseteq U$ with respect to which $f|_{K_i}$ is affine-linear. 
\end{definition}

Recall that a decomposition $U = \bigcup_{i \in I} K_i$ is {\em locally finite} if every point $x \in U$ has a neighborhood meeting finitely many $K_i$. A {\em PL homeomorphism} is a bijective PL function whose inverse is PL.

\begin{definition} A topological $n$--manifold $M$ is a PL $n$--manifold if it is equipped with an open covering $M = \bigcup_{i \in I} M_i$ and coordinate charts $\psi_i: M_i \rightarrow (U_i \subseteq \mathbb{R}^n)$ to open subsets of $\mathbb{R}^n$ for which all transition maps \[\psi_i \circ \psi_j^{-1}: \psi_j(U_i \cap U_j) \rightarrow \psi_i(U_i \cap U_j)\] are PL homeomorphisms.
\end{definition}

\begin{definition} \label{defn:finitePL} Let $M$ be a PL $n$--manifold. A continuous map $F: M \rightarrow \mathbb{R}^m$ is said to be {\em finite PL} if $M$ admits a finite polyhedral decomposition with respect to which $F$ is  affine-linear on cells
 with respect to the standard PL structure on the codomain, $\mathbb{R}^m$. 
That is, there is an imbedded finite polyhedral complex $\mathcal{C} \subseteq M$ and a compatible finite atlas $\{(M_i, \psi_i)\}_{i \in I}$ for $M$ for which 
\begin{itemize}
	\item $|\mathcal{C}| = M$, and
	\item for every $n$--cell $C$ of $\mathcal{C}$ there exists some $M_i$ such that $C \subset M_i$ and the map $(F \circ \psi_i^{-1})|_{\psi_i(C)}$ is affine-linear.
\end{itemize}
\end{definition}

\color{black}

\begin{remark} ReLU neural network functions $F: \mathbb{R}^n \rightarrow \mathbb{R}$ are finite PL with respect to the standard PL structures on $\mathbb{R}^n$ and $\mathbb{R}$. 
\end{remark}

We are now ready for the following definition, which is adapted from \cite[Defn. 4.1]{GrunertRote}.

\begin{definition} \label{defn:PLMorse} Let $M$ be a PL $n$--manifold and $F: M \rightarrow \mathbb{R}$ a finite PL map. 
\begin{itemize} 
	\item A point $p \in M$ is called {\em PL regular} for $F$ if there is a chart around $p$ such that the function $F$ can be used as one of the coordinates on the chart, i.e., if in those coordinates $F$ takes the form \[F_{\tiny reg}(x_1, \ldots, x_n) := F(p) +x_n.\]
	\item A point $p \in M$ is called {\em PL nondegenerate-critical} for $F$ if there is a PL chart around $p$ such that in those coordinates the function $F$ can be expressed as \[F_{\tiny crit,k}(x_1, \ldots, x_n) := F(p) -|x_1| - \ldots - |x_k| + |x_{k-1}| + \ldots + |x_n|.\] The number $k$ is then uniquely determined and is called the {\em index} of $p$. 
	\item If $p \in M$ is neither PL regular nor PL nondegenerate-critical, it is said to be \emph{PL degenerate-critical}.
	\item The function $F$ is called a {\em PL Morse function} if all points are either PL regular or PL nondegenerate-critical. 
\end{itemize}
\end{definition}

\begin{remark} If a point $p \in M$ is PL regular (resp., PL nondegenerate-critical of index $k$) for $F$, then the function $F$ is {\em locally-equivalent} at $p$ to $F_{\tiny reg}$ (resp., to $F_{\tiny crit, k}$) in the language of \cite[Defn. 3.1]{Grunert}. The interested reader may also wish to consult \cite[Sec. 3.2.5]{Grunert}, which reviews the relationship with previous notions due to Brehm-K{\"u}hnel, Eells-Kuiper, and Kosinski.
\end{remark}

\subsection{Flat cells, nontransversal thresholds, and general maps}

Recall the following terminology from \cite{GL}, which can be traced back to Banchoff \cite{Banchoff}:

\begin{definition} \label{defn:flatcell} Let $F: |M| \rightarrow \mathbb{R}$ be linear on the cells of a polyhedral complex, $M$.
A cell $C \in M$ satisfying $F(C) = \{t\}$ for some $t \in \mathbb{R}$ will be called a {\em flat} cell or a {\em constant} cell for $F$.
\end{definition}

\begin{remark} Note that all $0$--cells of $M$ are flat. \end{remark}

It is immediate that any face of a flat cell will also be flat, so the set of flat cells relative to a map $F: |M| \rightarrow \mathbb{R}$ forms a subcomplex of $M$, which we will denote $\MflatF$.

The following related definition, which appears frequently in the literature, is also due to Banchoff \cite{Banchoff}:

\begin{definition} \label{defn:GeneralPL} A map $F: |M| \rightarrow \mathbb{R}$ that is linear on cells of a polyhedral complex $M$ is said to be {\em general} if 
\begin{enumerate}
	\item $F(v) \neq F(w)$ for $0$--cells $v \neq w$ and 
	\item all flat cells are $0$--dimensional.
\end{enumerate}
\end{definition}

\begin{remark} \label{rmk:}  In \cite{Banchoff}, Banchoff restricts his attention to polytopal complexes and defines a map that is linear on cells of such a complex to be {\em general} if it satisfies condition (i) above, which implies condition (ii) in the bounded case. Since the complexes we consider will have unbounded cells, we add condition (ii). Note that we could equivalently have replaced condition (ii) with the condition that every unbounded cell has unbounded image. 
\end{remark}

\begin{remark} Note that while all flat cells of a general PL map are $0$-dimensional, not every PL map whose flat cells all have dimension $0$ is general. Note also that a general PL map need not be PL Morse. See, e.g., the example in \cite[Sec. 2]{Banchoff}.
\end{remark}

\begin{lemma} \label{lem:Flatncell} Let $M$ be a PL $n$--manifold, and let $F: M \rightarrow \mathbb{R}$ be a finite PL map. If $F$ has a flat cell $C$ of dimension $n$, then every point in $C$ is PL degenerate-critical, and hence $F$ is not PL Morse.
\end{lemma}

\begin{proof} By definition, if $p \in M$ is either PL regular or PL non-degenerate critical for $F$, and $F(p) = a$, then $p$ has a neighborhood $N_p$ that is PL homeomorphic to a neighborhood of the origin in $\mathbb{R}^n$, and the level sets of the induced map are PL homeomorphic. By examining $F_{=0}$ in each of the two local models, one concludes that the intersection of $F_{=a}$ with $N_p$ is therefore locally PL homeomorphic to an imbedded polyhedral complex of $N_p$ of dimension $n-1$. 

But if $p$ is contained in a flat cell of dimension $n$, then the intersection of $F_{=a}$ with every neighborhood of $p$ contains an $n$--dimensional ball, a contradiction.
\end{proof}

\color{black}
\begin{figure}
	\includegraphics[width=2in]{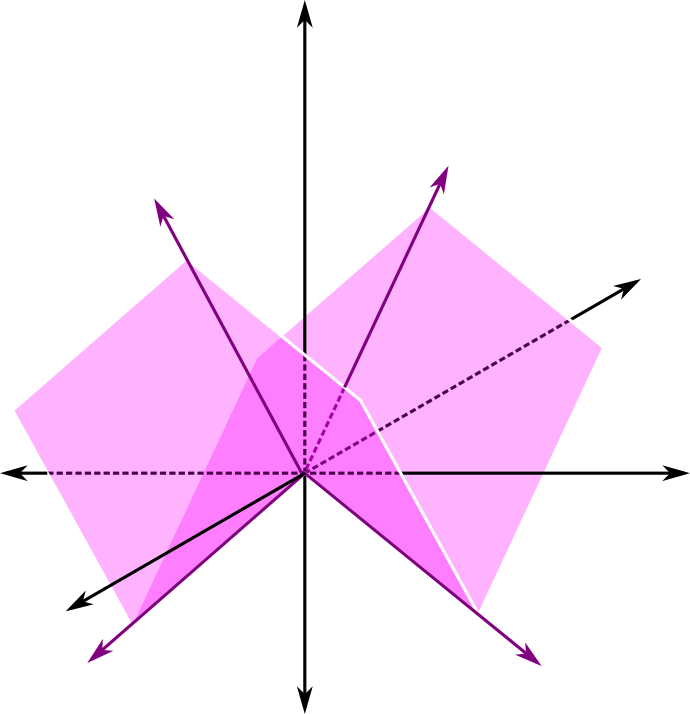}	\includegraphics[width=2in]{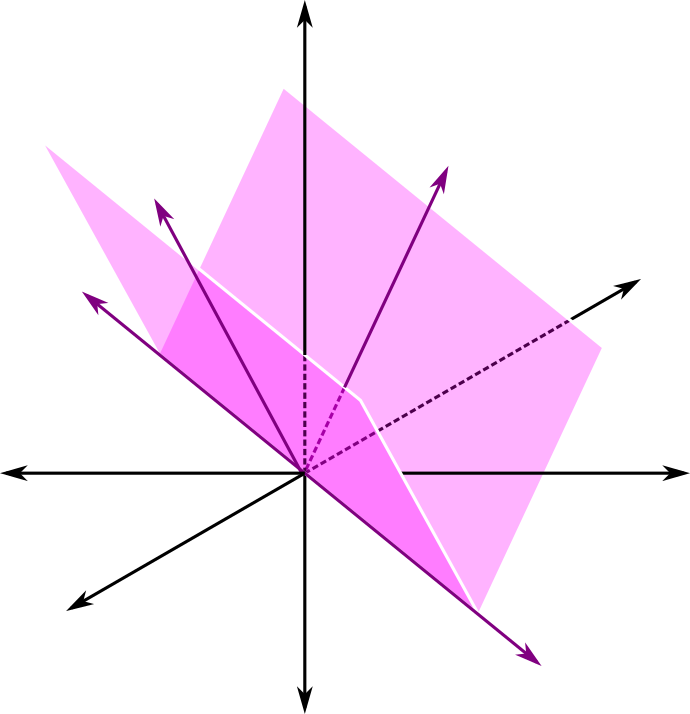}
	\caption{\textit{Left:} A PL non-degenerate critical point of index 1 of a function $\mathbb{R}^2 \to \mathbb{R}$. \textit{Right:} A PL strongly regular point.}
	\label{f:criticalpoints}
\end{figure}

\subsection{Flat cells for ReLU layer maps and proof of Theorem \ref{thm:NotPLMorse}}

To prove Theorem \ref{thm:NotPLMorse}, we briefly investigate how flat cells arise in ReLU neural network functions. 

It will be convenient to note that if $F$ is a ReLU neural network function of architecture $(n_0, \ldots, n_m; 1)$, 

then every cell $C$ of $\mathcal{C}(F)$ 
admits a labeling by a sequence of ternary tuples 
$$\theta_C := \left(\theta^{(1)}, \ldots, \theta^{(m)}\right) \in \{-1,0,1\}^{n_1+ \ldots + n_m}$$ indicating the sign of the pre-activation output of each neuron. 

Explicitly, if $x = x^{(0)} \in \mathbb{R}^{n_0}$ is an input vector, and we denote by 
$$z^{(\ell)}:= (A^{(\ell)} \circ F^{(\ell-1)} \circ \ldots \circ F^{(1)})(x) \in \mathbb{R}^{n_\ell}$$ the pre-activation output of the first $\ell$ layer maps, then the components of $\theta^{(\ell)}_C = \left(\theta^{(\ell)}_1,\,\, \ldots \,\,,\theta^{(\ell)}_{n_\ell}\right)$ are defined by $\theta^{(\ell)}_i = \mbox{sgn}(z_i^{(\ell)})$. For generic, transversal ReLU networks, this assignment agrees with the {\em activation patterns} defined in \cite{HaninRolnick}.

It is immediate that if $C$ is a cell for which the components of $\theta^{(\ell)}_C$ are all non-positive for some $\ell$, then $C$ is flat. The following lemma tells us that the converse holds for cells of generic, transversal neural networks with a single ReLU layer.

\begin{lemma} \label{l:onelayersendstopoint}
Let $F: \mathbb R^d \to \mathbb R^n$ be a generic ReLU neural network layer map with associated (generic) co-oriented hyperplane arrangement 
$\mathcal A :=\{H_1,... , H_n \}$.  Denote by $\CA$ the associated polyhedral decomposition of $\mathbb{R}^d$ and let $C\in\mathcal{C}(\mathcal{A})$ be a $k$-cell, for any $1\leq k \leq d$.  
If any coordinate of $F$ is constant on $C$ (i.e. $\pi_j \circ F(C)$ is a point),  then $F(C) = \vec{0}$. Furthermore, $C$ is a face of a $d$--cell with ternary labeling $(-1, \ldots, -1) \in \{-1,0,1\}^n$.
\end{lemma}
\color{black}

\begin{proof} 
First consider the case $1 \leq k < d$. 
	By the genericity of $\mathcal{A}$, the affine hull of $C$ is uniquely determined by some intersection of hyperplanes from $\mathcal A$, that is, $C \subset H_{i_1} \cap H_{i_2}... \cap H_{i_{d-k}}$. Suppose, by way of contradiction, that $C$ is sent to a nonzero constant in some coordinate corresponding to a hyperplane $H$ in $\mathcal{A}$. This precisely means that $C \in H^+$ and, additionally, for each $x \in C$, the distance $d(x, H)$ is constant and nonzero. Thus the affine hull of $C$ is contained in the hyperplane parallel to $H$ at a fixed distance.
	 This implies that $H \cap (H_{i_1}... \cap H_{i_{d-k}}) = \emptyset$.  However, this is an intersection of $d-k+1 \leq d$ hyperplanes in $\mathcal A$, and by genericity cannot be empty. This is a contradiction, and therefore if $C$ is sent to a constant, it must be sent to $0$ in all coordinates. Thus $C$ is in the intersection of the closed half spaces $H_i^-$  for all hyperplanes $H_i$ forming the arrangement $\mathcal{A}$, and must be a face of the cell with the ternary labeling $(-1,\ldots, -1)$. 
	 
	Now consider the case $k=d$, and again suppose that $C$ is sent to a nonzero constant in some coordinate corresponding to a hyperplane $H$ in $\mathcal{A}$.  Since $\mathcal{A}$ is nonempty, $C$ is not all of $\mathbb{R}^d$.  Hence $C$ contains a lower dimensional facet, $C'$.  Then $F$ sends $C'$ to the same nonzero constant in the coordinate corresponding to $H$.  But by the result of the previous case, $F(C') = \vec{0}$. 
	 	 Lastly, if $C$ were contained in $H_i^+$ for some $i$, then $F(C)$ would be nonzero in some coordinate, which is a contradiction. Thus $C$ has the ternary labeling $(-1,\ldots,-1)$.
	 \color{black}

\end{proof}

\begin{remark} \label{r:genericitynecessary}
The assumption in Lemma \ref{l:onelayersendstopoint} that the hyperplane arrangement is generic is necessary.  Here is an example of how the conclusion can fail without this assumption.  Consider the nongeneric layer map 
$F:\mathbb{R}^2 \to \mathbb{R}^2$ given by $$(x,y) \mapsto (\textrm{Relu}(x+1),\textrm{Relu}(x)).$$  Then $\{(x,y) : x = 0\}$ is a flat $1$-cell in $\mathcal{C}(F)$ whose image under $F$ is the point $(1,0)$, not $(0,0)$.  
 \end{remark}

\begin{remark}
The conclusion of Lemma \ref{l:onelayersendstopoint} does not always hold when $F$ is a composition of multiple generic, transversal layer maps. 

For example, define a neural network map  $F$ (omitting the final affine-linear layer) by \[F: \mathbb{R}^{2}  \xrightarrow{F_1= \sigma \circ A_1} \mathbb{R}^{2}  \xrightarrow{ F_2= \sigma \circ A_2}  \mathbb{R}^1\] where the affine-linear maps $A_i$ are defined by $A_1(x,y) = (x,y)$ and $A_2(u,v) = u+1$.  Then $\mathcal{C}(F) = \mathcal{C}^{std}(\mathbb{R}^2)$.  Let $C$ be the $1$-cell $\{(0,y) \in \mathbb{R}^2 \mid y \geq 0\}$ in $\mathcal{C}(F)$.  Then $C$ is flat since $F(C) = \{1\}$, but $C$ is not a face of a cell whose ternary labeling is $(-1, \ldots, -1)$ for either layer\color{black}. (Note that $\theta^{(2)} = 1$ for all points in the domain.)

Examples like this can arise when there exists some $i$ such that the image of a positive-dimensional cell, $C$, of $\mathcal{C}(F)$ under the first $i-1$ layer maps, $F^{(i-1)} \circ \ldots \circ F^{(1)}$, 
\begin{enumerate}
	\item is in the orthogonal complement of the span of the weight vectors of $F^{(i)}$, and
	\item at least one neuron of $F^{(i)}$ is switched on in the interior of $C$.
\end{enumerate} 
\end{remark}
\color{black}

The following was essentially proved in \cite{Grunert}:

\begin{lemma} \label{l:IntNonFlatRegular} Let $F: \mathbb{R}^n \rightarrow \mathbb{R}$ be a ReLU neural network function and $\mathcal{C}(F)$ its canonical polyhedral complex. Any point $p$ in the interior of a non-flat cell $C$ of dimension $k \geq 1$ is PL regular. 
\end{lemma}

\begin{proof} Recall that $F$ is affine-linear on the cells of $\mathcal{C}(F)$ by construction. Let $p$ be a point in the interior of a cell $C$ of dimension $k \geq 1$. Any finite simplicial subdivision of the intersection complex, $\mathcal{C}_p$, of a star neighborhood of $p$ with  $\mathcal{C}(F)$ will be a combinatorial $n$--manifold with boundary, since it is imbedded in $\mathbb{R}^n$ (cf. the discussion following \cite[Defn. 1.33]{Grunert}).
If $C$ is non-flat, it then follows from \cite[Lem. 3.22]{Grunert} that $p$ is PL regular.
\end{proof}
\color{black}
We are now ready to prove Theorem \ref{thm:NotPLMorse}.

\begin{proof}[Proof of Theorem \ref{thm:NotPLMorse}] 
Let $F: \mathbb{R}^n \rightarrow \mathbb{R}$ be a ReLU neural network map with hidden layers of dimensions $(n_1, \ldots, n_\ell)$.

We focus first on proving that when $\ell = 1$ the probability that $F$ is PL Morse is \[\frac{\sum_{k=n+1}^{n_1} {n_1 \choose k}}{2^{n_1}}.\] 
Note that since the output layer map $G: \mathbb{R}^{n_1} \rightarrow \mathbb{R}$  is affine-linear, in the $\ell = 1$ case we have $\mathcal{C}(F)=\mathcal{C}(F^{(1)})$ where $F^{(1)}: \mathbb{R}^n \rightarrow \mathbb{R}^{n_1}$ is the first layer map. In other words, $\mathcal{C}(F) = \mathcal{C}(\mathcal{A})$, where $\mathcal{A}$ is the hyperplane arrangement associated to $F^{(1)}$.

We now claim that the probability that $F$ is PL Morse when $\ell = 1$ is precisely the probability that no $n$--cell of $\mathcal{C}(F)$ has ternary labeling $(-1, \ldots, -1)$. 

To prove this claim, we first note that since there are finitely many cells of $\mathcal{C}(F)$, and the final layer map $G$ is affine-linear, the probability that $\mathcal{C}(F)$ has a flat cell of positive dimension is precisely the probability that $\mathcal{C}(F^{(1)})$ does. Moreover, Lemma \ref{l:onelayersendstopoint} tells us that $\mathcal{C}(F^{(1)})$ has a flat cell of positive dimension if and only if it has a flat cell of dimension $n$ with ternary labeling $(-1, \ldots, -1)$.

If there is an $n$--cell of $\mathcal{C}(F)$ with ternary labeling $(-1, \ldots, -1)$, Lemma \ref{lem:Flatncell} tells us that $F$ is {\em not} PL Morse. If there is no $n$--cell with ternary labeling $(-1, \ldots, -1)$, then the only flat cells of $F$ are $0$--cells. Lemma \ref{l:IntNonFlatRegular} above tells us that any point in the interior of a non-flat cell is PL regular, and Proposition \ref{p:stdmodel} tells us that every $0$--cell of $\mathcal{C}(F)$ is PL regular or PL nondegenerate-critical, since there are no flat cells of positive dimension. It follows that $F$ is PL Morse, completing the proof of the claim.

To compute the probability that an $n$--cell of $\mathcal{C}(F) = \mathcal{C}(\mathcal{A})$ has ternary labeling $(-1, \ldots, -1)$, recall that Zaslavsky's theorem (cf. \cite[Prop. 2.4]{Stanley}) tells us that the total number of $n$--cells of $\mathcal{C}(\mathcal{A})$ when $\mathcal{A}$ is a generic arrangement of $n_1$ hyperplanes in $\mathbb{R}^n$ is $r(\mathcal{A}) = \sum_{k=0}^n {n_1 \choose k}$. Associated to each region is precisely one co-orientation of the hyperplane arrangement causing the region to be labeled with $\vec{0}$. So, of the $2^{n_1}$ possible co-orientations on each generic arrangement, precisely $r(\mathcal{A}) = \sum_{k=0}^n {n_1 \choose k}$ yield an $n$--cell with ternary labeling $(-1, \ldots, -1)$. But $\mathcal{A}$ is generic for all but a measure $0$ set of parameters. It follows that the probability that $F$ is PL Morse is \[1 - \frac{\sum_{k=0}^n {n_1 \choose k}}{2^{n_1}} = \frac{\sum_{k=n+1}^{n_1} {n_1 \choose k}}{2^{n_1}},\] as desired.

We will now prove that the above expression gives an {\em upper bound} on the probability that $F: \mathbb{R}^n \rightarrow \mathbb{R}$ is PL Morse in the case $\ell > 1$. By the discussion above and before the statement of Lemma \ref{l:onelayersendstopoint}, the probability that $F$ has a flat $n$--cell is bounded below by the probability that $\mathcal{C}(F^{(1)})$ has a region with ternary label $(-1, \ldots, -1)$. The latter probability is $\frac{\sum_{k=0}^n {n_1 \choose k}}{2^{n_1}}$, as above, so \[\frac{\sum_{k=n+1}^{n_1} {n_1 \choose k}}{2^{n_1}}\] is an {\em upper bound} on the probability that $F$ is PL Morse.
\end{proof}

\begin{corollary} \label{cor:goeszero} Let $F: \mathbb{R}^n \rightarrow \mathbb{R}$ be a ReLU neural network function with hidden layers of dimensions $(n_1, \ldots, n_\ell)$. 
\begin{enumerate}
	\item If $n_{1} \leq n$, the probability that $F$ is PL Morse is $0$.
	\item If $\ell = 1$ and $n_1  \geq 2n+1$, the probability that $F$  is PL Morse exceeds $\frac{1}{2}$.
	\item If $\ell = 1$, the probability that $F$ is PL Morse approaches $1$ as $n_1 \rightarrow \infty$.
\end{enumerate}
\end{corollary}

\begin{proof} Item (i) is immediate. To see (ii), recall that $\sum_{k=0}^{n_1} {n_1 \choose k} = 2^{n_1}$ and ${n_1 \choose k} = {n_1 \choose n_1 - k}$. To see (iii), note that for fixed $n$, $\sum_{k=0}^n {n_1 \choose k}$ is a degree $n$ polynomial in $n_1$, while $2^{n_1}$ is exponential in $n_1$, so \[\lim_{n_1 \rightarrow \infty} \frac{\sum_{k=0}^n {n_1 \choose k}}{2^{n_1}} \rightarrow 0.\]
\end{proof}

\begin{definition} Let $F: \mathbb{R}^n \rightarrow \mathbb{R}$ be linear on cells of a polyhedral decomposition $\mathcal{C}$ of $\mathbb{R}^n$. We say that $F$ is {\em almost PL Morse} if there exists some cell $C \in \mathcal{C}$ such that if $x$ is neither PL regular nor PL nondegenerate critical, then $x \in |C|$.
\end{definition}

The following is immediate from the proof of Theorem \ref{thm:NotPLMorse}.

\begin{corollary}  \label{t:almostMorseProbab}
Let $F: \mathbb{R}^n \rightarrow \mathbb{R}^{n_1} \rightarrow \mathbb{R}$ be a depth $2$ ReLU neural network function. $F$ is almost PL Morse with probability $1$.
\end{corollary}

We close the section with a first estimate of the prevalence of flat cells, in the form of a lower bound.

\begin{lemma}
Fix $x$ in $\mathbb{R}^{n_0}$, and a neural network architecture $(n_0, n_1,... n_m)$. Select a network $F$ with this architecture randomly such that its parameters are i.i.d. following a distribution which is symmetric about the origin. Letting $C$ be the minimal-dimensional cell in $\mathcal{C}(F)$ containing $x$, then the probability that $C$ is a flat cell for $F$ is at least $2^{-n_m}$. 
\end{lemma}

\begin{proof}
	We may assume that $C$ is an $n_0$-cell, as the probability that $x$ is in the $(n_0-1)$-skeleton of $\mathcal{C}(F)$ is zero. We also assume that the network is transversal so that the bent hyperplane arrangement forms the $(n_0 - 1)$ skeleton of $\mathcal{C}(F)$. 
	
	Recall $\mathcal{A}_m$, the hyperplane arrangement associated with $\mathbb{F}_m$, and its associated co-oriented hyperplane arrangement. Now, $C$ is given by the intersection of two cells in $\mathbb{R}^{n_0}$: A unique cell $R' = (F_m^{-1} \circ ... \circ F_1)^{-1}(R)$, with $R$ a top-dimensional region of the polyhedral decomposition of $\mathbb{R}^{n_{m-1}}$ given by $\mathcal{A}_m$, together with a unique cell of $\mathcal{C}(F_{m-1}\circ ... \circ F_1)$.
	
	Because the parameters for $F$ are selected i.i.d. from a distribution which is symmetric about the origin, the following construction of the network $F'$ has the same distribution as $F$: Selecting a network from the same distribution as $F$, then randomly and independently multiplying each row of $A_m$ by -1 or 1 with equal probability, effectively selecting $F$ and then uniformly selecting a co-orientation of the hyperplane arrangement $\mathcal{A}_m$ to produce $F'$. We note that the polyhedral complex $\mathcal{C}(F)$ is equivariant under changing co-orientations of $\mathcal{A}_m$: all such co-orientations lead to the same final polyhedral complex $\mathcal{C}(F)$ and the presence of the same cell $C$. Recall that there is a unique co-orientation of $\mathcal{A}_m'$ making $F_m$ constant on $R$. There are $2^{n_m}$ such co-orientations, so the probability that $F_m'$ is constant on $R$, given initial network $F$, is $2^{-n_m}$. If $F_m'$ is constant on $R$, then $F'$ is constant on $C$, so the probability $F'$ is constant on $C$ given $F$ is at least $2^{-n_m}$. Since this is true for almost all $F$, the probability any $F'$ selected in this way is constant on the cell containing $x$ is at least $2^{-n_m}$. As $F'$ follows the same distribution as $F$, this is true for any network $F$ as well. 
\end{proof}
\color{black}

\section{Global and Local H-complexity} \label{sec:Hcompxty}

The purpose of this section is to define and study algebro-topological invariants of the level sets and sublevel sets of ReLU neural network functions. We extract local and global notions of topological complexity of ReLU neural network functions by studying how the homology of its level sets and sublevel sets change as we vary the final layer threshold. 

We recall the following definitions from algebraic topology (cf. \cite{Hatcher}) and piecewise linear (PL) topology (cf. \cite{RS}):

\begin{definition} Let $X,Y$ be topological spaces, and let $I := [0,1]$ denote the closed unit interval endowed with the standard topology. A {\em homotopy} between continuous maps $f,g: X \rightarrow Y$ is a continuous map $H: X \times [0,1] \rightarrow Y$ satisfying $H|_{X \times \{0\}} = f$ and $H|_{X \times \{1\}} = g$. 
\end{definition}

If there exists a homotopy between $f,g: X \to Y$ we say that $f$ and $g$ are {\em homotopic} and write $f \simeq g$.

\begin{definition} \label{defn:HE} Let $X,Y$ be topological spaces. $X$ and $Y$ are said to be {\em homotopy equivalent} if there exist continuous maps $f: X \rightarrow Y$ and $g: Y \rightarrow X$ such that $f\circ g \simeq \mbox{Id}$ and $g \circ f \simeq \mbox{Id}$. In this case, $f$ (resp., $g$) {\em induces a homotopy equivalence} between $X$ and $Y$. \color{black}
\end{definition}

 \emph{Strong deformation retraction} is a special case of homotopy equivalence that will be particularly important in the present work. A strong deformation retraction of a space $X$ onto a subspace $A \subset X$ is a continuous map $F:X \times [0,1] \to X$ such that $F(x,0)=x$, $F(x,1) \in A$ and $F(a,t)=a$ for all $x \in X, a \in A, t \in [0,1]$. 

Homotopy equivalence is an equivalence relation on topological spaces that is weaker than homeomorphism. Many algebro-topological invariants (in particular, the ones studied here) are homotopy equivalence class invariants. 

Since we will also define invariants of pairs $(X,A)$ of topological spaces where $A \subseteq X$ is a subspace (i.e., the inclusion map is continuous), we will need a few more definitions. 

\begin{definition} Let $(X,A)$ and $(Y,B)$ be pairs of topological spaces. A {\em map on pairs} $f: (X,A) \rightarrow (Y,B)$ is a continuous map $X \rightarrow Y$ whose restriction to $A$ is a continuous map $A \rightarrow B$.
\end{definition}

\begin{definition} Let $f, g: (X,A) \rightarrow (Y,B)$ be maps of pairs. $f$ and $g$ are said to be {\em homotopic through maps of pairs} $(X,A) \rightarrow (Y,B)$ if there exists a homotopy $H: X \times I \rightarrow Y$ for which $H_t: X \times \{t\} \rightarrow Y$ is a map on pairs $(X,A) \rightarrow (Y,B)$ for all $t \in [0,1]$.
\end{definition}

Recall that if $(X,A)$ is a pair of topological spaces, then $H_n(X,A)$ refers to the degree $n$ relative singular homology of the pair $(X,A)$ with $\mathbb{Z}$ coefficients, cf. \cite[Chp. 2]{Hatcher}. If $X$ is a simplicial complex and $A$ is a subcomplex, then the relative singular homology agrees with the relative simplicial homology of the pair $(X,A)$, cf. \cite[Thm. 2.27]{Hatcher}. 

Note that if $f: (X,A) \rightarrow (Y,B)$ is a map of pairs then $f$ induces a homomorphism $f_*: H_n(X,A) \rightarrow H_n(Y,B)$ for all $n$. We will make use of the following standard results, cf. \cite[Prop. 2.19, Ex. 2.27]{Hatcher}.

\begin{lemma} 
If two maps $f,g$ are homotopic through maps of pairs $(X,A) \rightarrow (Y,B)$, then $f_*=g_*: H_n(X,A) \rightarrow H_n(Y,B)$ for all $n$. 
\end{lemma}

\begin{lemma} If $f: (X,A) \rightarrow (Y,B)$ is a map on pairs for which $f: X \rightarrow Y$ and its restriction $f|_A: A \rightarrow B$ induce homotopy equivalences, \color{black} then $f_*: H_n(X,A) \rightarrow H_n(Y,B)$ is an isomorphism for all $n$.
\end{lemma}

The following definition is adapted from \cite{GrunertRote}. 

\begin{definition} \label{defn:HregHcrit} Let $F: |M| \rightarrow \mathbb{R}$ be a map that is linear on cells of a finite polyhedral complex $M$, let $a \in \mathbb{R}$ be a {\em nontransversal} threshold, and $K$ a connected component of the intersection of its corresponding level set, $F_{=a}$, with the subcomplex, $\MflatF$, of flat cells. $K$ is said to be {\em homologically critical} (or {\em H-critical} for short) if $H_*(F_{\leq a}\,\,, F_{\leq a} \setminus K) \neq 0$. $K$ is said to be {\em homologically regular} (resp., H-regular) otherwise.
\end{definition}

\begin{remark} In \cite{GrunertRote}, the authors define the notion of H-criticality and H-regularity only for flat $0$--cells and not for (connected unions of) flat higher-dimensional cells. Since the functions of interest to us have a non-zero probability of having higher-dimensional flat cells, we extend the definition. 
\end{remark}

\begin{definition} \label{defn:Hcpxty} Let $F: |M| \rightarrow \mathbb{R}$ be a map that is linear on cells of a finite polyhedral complex $M$, let $a \in \mathbb{R}$ be a {\em nontransversal} threshold, and let $K$ be a connected component of $F_{=a} \cap \MflatF$ as in Definition \ref{defn:HregHcrit} above. 
\begin{enumerate}
	\item The {\em $i$th local $H$--complexity} of $K$ is defined to be the rank of \[H_i(F_{\leq a}\,\, , F_{\leq a} \setminus K).\]
	\item The {\em total local $H$--complexity} of $K$ is defined to be the sum of the $i$th local $H$--complexities of $K$ over all $i$.
	\item The {\em global $H$--complexity} of $F$ is the sum, over all nontransversal thresholds $a \in \mathbb{R}$ and all connected components $K$, of \[F_{=a} \cap \MflatF,\] of the total local $H$--complexity of $K$.
\end{enumerate}
\end{definition}

\begin{remark} Note that the universal coefficient theorem tells us that if the total local $H$--complexity $H_*(F_{\leq a}\,\, , F_{\leq a} \setminus K) = 0$ then $H_*(F_{\leq a}\,\, , F_{\leq a} \setminus K; \mathbb{F}) = 0$ for any field $\mathbb{F}$, but this statement may not hold for the $i$th local $H$--complexity for fixed $i$.
\end{remark}

\color{black}

\begin{proposition} Let $M$ be a PL $n$--manifold and $F: M \rightarrow \mathbb{R}$ be a finite PL map. If $v$ is either a PL regular or PL nondegenerate-critical $0$--cell, then $v$ is  a connected component of $F_{=a} \cap \MflatF$. Moreover:
\begin{enumerate}
	\item If $v$ is PL regular, then the $i$th local $H$--complexity of $v$ is $0$ for all $i$.
	\item If $v$ is PL nondegenerate-critical of index $k$ with $F(v) = a$, then the $i$--th local H-complexity of $v$ is 
$\left\{\begin{array}{cl} 1 & \mbox{ if $i=k$ and}\\
0 & \mbox{ otherwise.}\end{array}\right.$
\end{enumerate}
\end{proposition}

\begin{proof} It follows immediately from the definitions of regular and non-degenerate critical point that $v$ is {\em isolated}, so $v$ is a connected component of $F_{=a} \cap \MflatF$. Consider the intersection complex of $M$ with any closed star neighborhood of $v$ containing no other flat cells. In \cite[Thm. 5.2]{Grunert}, Grunert proves that if 
\begin{itemize}
	\item $M$ is a polytopal complex with $F: M \rightarrow \mathbb{R}$ linear on cells, 
	\item $v$ is a non-degenerate critical point of index $k$ with $F(v) = a$, and 
	\item $v$ is the only $0$--cell in $F^{-1}[a-\epsilon,a+\epsilon]$ for some $\epsilon > 0$,
\end{itemize}
then $F_{a+\epsilon}$ is homotopy-equivalent to $F_{a-\epsilon}$ with a $k$--handle attached. By applying the above result to the restriction of $F$ to the star complex of $v$ in the intersection complex constructed above, it follows by excision (cf. the discussion following \cite[Defn. 3.3]{GrunertRote}) that \[H_i(F_{\leq a} \,\, ,F_{\leq a} \setminus \{v\}) = \left\{\begin{array}{cl} \mathbb{Z} & \mbox{ if $i=k$}\\
0 & \mbox{ otherwise. }\end{array}\right.\]
\end{proof}
\color{black}

\section{Sublevel sets of finite PL maps are homotopy equivalent within transversal intervals}

The goal of this section is to develop the necessary technical machinery to prove the following two keystone results of this paper:

\begin{retractiontheorem} 
Let $\mathcal{C}$ be a polyhedral complex in $\mathbb{R}^n$, let $F:|\mathcal{C}| \to \mathbb{R}$ be linear on cells of $\mathcal{C}$, and  let $[a,b]$ be an interval of transversal thresholds for $F$ on $\mathcal{C}$.  Then there exists a polytopal complex $\mathcal{D}$ with $|\mathcal{D}| \subset |\mathcal{C}|$ and a strong deformation retraction $\Phi:|\mathcal{C}| \to |\mathcal{D}|$ (which induces a cellular map $\mathcal{C} \to \mathcal{D}$) such that
$\Phi$ is $F$-level preserving on $|\mathcal{C}_{F \in [a,b]}|$ and for every $c \in [a,b]$,
\begin{equation} \label{eq:PhiProperties}
\Phi( |\mathcal{C}_{F \leq c}| ) = |\mathcal{D}_{F \leq c}|, \quad \Phi( |\mathcal{C}_{F \geq c}| ) = |\mathcal{D}_{F \geq c}|.
\end{equation}
Furthermore, there is an $F$-level preserving PL isotopy  
$$\phi: |\mathcal{D}|_{F = c} \times [a,b] \to |\mathcal{D}_{F \in [a,b]}|.$$
\end{retractiontheorem}

\begin{homotopyequivtheorem}
Let $\mathcal{C}$ be a polyhedral complex in $\mathbb{R}^n$, let $F:|\mathcal{C}| \to \mathbb{R}$ be linear on cells of $\mathcal{C}$, and let $[a,b]$ be an interval of transversal thresholds for $F$ on $\mathcal{C}$.  Then for any threshold $c \in [a,b]$, $|\mathcal{C}_{F \leq c}|$ is homotopy equivalent to $|\mathcal{C}_{F \leq a}|$; also $|\mathcal{C}_{F \geq c}|$ is homotopy equivalent to $|\mathcal{C}_{F \geq b}|$. \end{homotopyequivtheorem}

While we believe the above results are useful more generally, our main motivation for proving them here is to extend existing results in the literature adapting Morse-theoretic ideas to the PL setting so that they apply to some noncompact settings. Indeed, Theorem \ref{t:retraction} allows us to extend \cite[Lem. 4.13]{Grunert}, which applies only to polytopal complexes, to all polyhedral complexes imbedded in $\mathbb{R}^n$ equipped with the standard PL structure. 

Note also that Theorem \ref{t:homotopyequiv} is sufficient to guarantee that for a map that is affine-linear on cells of a finite polyhedral complex in $\mathbb{R}^n$, we need only perform computations at finitely many thresholds $a \in \mathbb{R}$ in order to compute the full set of algebro-topological invariants described in Section \ref{sec:Hcompxty}. In addition, we note that having compact models for the spaces of interest guarantees the existence of finite time algorithms to compute the algebro-topological invariants we study, since polytopal complexes admit finite simplicial subdivisions, and simplicial homology can be computed algorithmically.

This section is organized as follows. In Subsection \ref{s:canonpolytopalcpx}, we establish terminology and notation. In particular, we recall the classical decomposition of polyhedral sets as a Minkowski sum (Theorem \ref{t:polyhedraldecomposition}). In Subsection \ref{s:polytopalcpxpointed} we use this decomposition in the {\em pointed} case (see Definition \ref{d:SchrijverDefs})  to construct a deformation retraction to a canonical polytopal complex. In Subsection \ref{s:essentialization} we prove that for connected polyhedral complexes imbedded in $\mathbb{R}^n$, all cells are pointed or all cells are unpointed. This allows us to extend the deformation retraction to the unpointed setting. In Subsection \ref{s:homequivsublevel}, we then show that, with respect to any map that is affine-linear on cells, these deformation retractions are well-behaved enough to guarantee that sublevel sets are homotopy equivalent within transversal intervals.

\color{black}
\subsection{The canonical polytopal complex of a pointed polyhedral complex} \label{s:canonpolytopalcpx}

We begin by recalling some definitions from \cite{Schrijver} pertaining to the structure of polyhedral sets:

\begin{definition}[\cite{Schrijver}] \label{d:SchrijverDefs} \ 
\begin{enumerate} 
\item A \emph{cone} in $\mathbb{R}^n$ is a set $C$ such that if $x,y \in C$ and $\lambda, \mu \geq 0$, then $\lambda x + \mu y \in C$.  
\item Let $P$ and $Q$ be polyhedral sets embedded in $\mathbb{R}^n$.  The \emph{Minkowski sum} of $P$ and $Q$ is 
$$P + Q := \{p+q \mid p \in P, q \in P\}.$$
\end{enumerate}
 Let $P \subset \mathbb{R}^n$ be a polyhedral set.  
\begin{enumerate}
\setcounter{enumi}{2}
\item The \emph{characteristic cone} (or \emph{recession cone}) of $P$ is the set 
$$\textrm{char.cone}(P) := \{ y \mid x+y \in P \textrm{ for all }x \in P\}$$
\item The \emph{lineality space} of $P$ is the set 
$$\textrm{lin.space}(P) := \textrm{char.cone}(P) \cap - \textrm{char.cone}(P).$$
\item $P$ is said to be \emph{pointed} if $\textrm{lin.space}(P) = \{0\}$. 
\item A \emph{minimal face} of $P$ is a face not containing any other face.  
\item A \emph{minimal proper} face of a cone $C$ is a face of dimension 
$$\textrm{dim}(\textrm{lin.space}(C)) + 1.$$
\end{enumerate}
\end{definition}

Recall also that the \emph{convex hull} (resp. \emph{linear hull} and \emph{affine hull}) of a set $S \subset \mathbb{R}^n$ is the intersection of all convex subsets (resp. linear and affine-linear subspaces) of $\mathbb{R}^n$ that contain $S$.

The following well-known characterization of pointed cells is immediate:

\begin{lemma}[\cite{Schrijver}] \label{l:ptdequalsvertex} Let $P \subset \mathbb{R}^n$ be a polyhedral set. $P$ is pointed if and only if it has a face of dimension $0$.
\end{lemma}

It will be important to us that a pointed polyhedral set has a canonical representation as the Minkowski sum of its compact part and its characteristic cone:

\begin{theorem} \cite[Theorem 8.5]{Schrijver} \label{t:polyhedraldecomposition}
Let $P \subset \mathbb{R}^n$ be a polyhedral set.  For each minimal face $F$ of $P$, fix a point $x_F \in F$. For each minimal proper face $F$ of $\textrm{char.cone}(P)$, fix a point $y_F \in F \setminus \textrm{lin.space}(P)$.  Fix a set $\{z_1,\ldots,z_k\}$ that generates the linear space $\textrm{lin.space}(P)$. Then 
\begin{multline*}
P = \textrm{conv.hull}(\{x_F \mid F \textrm{ a minimal face of }P\})  + \\ 
\textrm{cone} (\{y_F \mid F \textrm{ a minimal proper face of }\textrm{char.cone}(P) \}) + 
\\ \textrm{lin.hull}(\{z_1,\ldots,z_k\})
\end{multline*}
\end{theorem}

Equipped with the definition of a pointed polyhedral set, we can now define and establish notation for its compact part:

\begin{definition}[Compact part of a pointed polyhedral set]
For any polyhedral set $C$ embedded in $\mathbb{R}^n$, we denote by $K_C$ the convex hull of the set of $0$-cells of $C$, and denote by $\widetilde{K}_C$ the polytopal complex consisting of all the faces of $K_C$ (including the set $K_C$ itself). 
\end{definition}

\begin{definition}[Compact part of pointed complexes in $\mathbb{R}^n$] \label{def:compactpart}
Let $\mathcal{C}$ be a polyhedral complex embedded in $\mathbb{R}^n$ in which all cells are pointed.  
 Define the \emph{canonical polytopal complex associated to} $\mathcal{C}$ to be the following family of sets:
 \begin{equation}  \label{eq:union} \textrm{com.part}(\mathcal{C}) :=  \bigcup_{C \in \mathcal{C}} \widetilde{K}_C.\end{equation}

\end{definition}

\begin{remark}
We interpret the union in \eqref{eq:union} as not allowing redundancy, i.e. each set occurs only once in $\textrm{com.part}(\mathcal{C})$, even if it is represented in the union more than once.  
\end{remark}

For a polyhedral complex $\mathcal{C}$ whose cells are all pointed, we will also refer to $\textrm{com.part}(\mathcal{C})$ as the canonical polytopal complex associated to $\mathcal{C}$ (c.f. Definition \ref{d:canonicalpolytopalcpx}).  
 Calling $\textrm{com.part}(\mathcal{C})$ a ``polytopal complex'' is justified by the following result:

\begin{lemma} \label{l:polytopalcpxiscpx}
Let $\mathcal{C}$ be a polyhedral complex embedded in $\mathbb{R}^n$ in which all cells are pointed.  Then the compact part of $\mathcal{C}$, $\textrm{com.part}(\mathcal{C})$, is a polytypal complex.
\end{lemma}

\begin{proof}

	A polyhedral complex, by definition, has only finitely many $0$-cells, so for each cell $C \in \mathcal{C}$, $K_C$ is a polytope.  Furthermore, a face of a polytope is a polytope.  Hence, $\textrm{com.part}(\mathcal{C})$ is a finite collection of polytopes.  We must verify that
	\begin{enumerate}
	\item \label{i:verify1} every face of each polytope in $\textrm{com.part}(\mathcal{C})$ is in  $\textrm{com.part}(\mathcal{C})$, and 
	\item \label{i:verify2} the intersection of any two polytopes in  $\textrm{com.part}(\mathcal{C})$ is a face of both.  
	\end{enumerate}
	
	Property \eqref{i:verify1} holds because  $\textrm{com.part}(\mathcal{C})$ consists of a union of polytopal complexes, and each such complex must satisfy \eqref{i:verify1}.  
	
	To see property \eqref{i:verify2}, suppose that $A$ and $B$ are any two polytopes in $\textrm{com.part}(\mathcal{C})$. We wish to show that $A \cap B$ is additionally a member of $\textrm{com.part}(\mathcal{C})$, and is a face of both $A$ and $B$. Let $C_1$ and $C_2$ be minimal cells in $\mathcal{C}$ such that $A \in \widetilde{K}_{C_1}$ and $B \in \widetilde{K}_{C_2}$. We claim that $A \cap B$ is a face of $K_{C_1 \cap C_2}$, and therefore a member of $\textrm{com.part}(\mathcal{C})$. Furthermore, $A \cap B$ is a face of both $A$ and $B$. 
	
	First, we may assume $A \cap B$ is nonempty, since otherwise \eqref{i:verify2} holds vacuously. Then, since any two cells of $\mathcal{C}$ are either disjoint or intersect in a common face, we know $C_1 \cap C_2$ is a polyhedral set in $\mathcal{C}$, and a face of both $C_1$ and of $C_2$.  As a result, the $0$-cells of $C_1 \cap C_2$ are precisely those $0$--cells contained in both $C_1$ and $C_2$. Two compact polytopes are equal as sets if and only if they have the same vertices, as taking convex hulls preserves subsets.
	This implies $K_{C_1 \cap C_2} = K_{C_1} \cap K_{C_2}$. 
	
	Furthermore, $K_{C_1 \cap C_2}$ is a face of $K_{C_1}$. We see this as follows. Since $C_1 \cap C_2$ is a face of $C_1$, there is a supporting hyperplane $H$ of $C_1$ such that $C_1 \cap C_2 = C_1 \cap H$. Consider $K_{C_1} \cap H$. First, $H$ is a supporting
	hyperplane of $K_{C_1}$ since $H$ does not cut $C_1$, so it cannot cut $K_{C_1}\subset  C_1$. Since $K_{C_1 \cap C_2}\subset (C_1 \cap H)$ is nonempty, $K_{C_1} \cap H$ is nonempty, and must be a face of $K_{C_1}$. We claim $K_{C_1}\cap H = K_{C_1 \cap C_2}$. Now, since $K_{C_1}\cap H$ is a face of $K_{C_1}$, the vertices of $K_{C_1}\cap H$ are precisely the vertices of $K_{C_1}$ contained in $H$, which are precisely the vertices of $C_1$ contained in $H$. But the vertices of $C_1$ contained in $H$ are the vertices of $C_1 \cap H$, since $C_1 \cap H$ is a face of $C_1$. Since $C_1 \cap H = C_1 \cap C_2$, the vertex set of of $C_1 \cap H$ is precisely the same as the vertex set of $C_1 \cap C_2$, which is the vertex set of $K_{C_1 \cap C_2}$. Thus $K_{C_1} \cap H = K_{C_1 \cap C_2}$ since their vertex sets are equal, and $K_{C_1 \cap C_2}$ is a face of $K_{C_1}$.  By symmetry, $K_{C_1 \cap C_2}$ is also a face of $K_{C_2}$.

	Now, to see that $A \cap B$ is a face of $K_{C_1 \cap C_2}$, we note that $A \cap K_{C_1 \cap C_2}$ must be a face of $K_{C_1 \cap C_2}$ since both $A$ and $K_{C_1 \cap C_2}$ are faces of $K_{C_1}$. Likewise $B \cap K_{C_1 \cap C_2}$ must be a face of $K_{C_1 \cap C_2}$. Since $\tilde{K}_{C_1 \cap C_2}$ is a polytopal complex, their intersection, $A \cap B \cap K_{C_1 \cap C_2}$ must be a face of $K_{C_1 \cap C_2}$ as well. But $A \cap B \subset K_{C_1}\cap K_{C_2}$, so $A\cap B \cap K_{C_1 \cap C_2} = A\cap B$, and $A \cap B$ must be a face of $K_{C_1 \cap C_2}.$

	This establishes that $A \cap B$ is a member of $\textrm{com.part}(\mathcal{C})$.

	Lastly, we show that $A \cap B$ is a face of $A$. However this is immediate. We note that $A$ is a face of $K_{C_1}$, $K_{C_1 \cap C_2}$ is a face of $K_{C_1}$, and $A \cap B$ is a face of $K_{C_1 \cap C_2}$, and thus a face of $K_{C_1}$. Since $\tilde{K}_{C_1}$ is a polytopal complex, two faces of $K_{C_1}$ intersect in a common face if the intersection is nonempty, telling us that $A \cap (A \cap B)$ is a face of $A$. By symmetry, $B \cap (A \cap B)$ is a face of $B$. 	 
\end{proof}

\begin{remark}
	We must extend the argument to faces of $K_{C}$ because if a cell $C$ is unbounded,$\tilde{K}_{C}$ may have a face which cannot be expressed as $K_{C'}$ for any cell $C'$ in $\mathcal{C}$. 
\end{remark}

\color{black}
\subsection{Deformation retraction of a pointed polyhedral complex of $\mathbb{R}^n$ to its canonical polytopal complex} \label{s:polytopalcpxpointed}

\color{black}

The purpose of this subsection is to prove:

\begin{proposition}\label{p:pointeddeformationretraction}
Let $\mathcal{C}$ be a polyhedral complex in $\mathbb{R}^n$ in which all cells are pointed. Then there is a strong deformation retraction from $|\mathcal{C}|$ to  $|\textrm{com.part}|(\mathcal{C})|$ that induces a surjective cellular map from $\mathcal{C}$ to $\textrm{com.part}(\mathcal{C})$ respecting the face poset partial order.
\end{proposition} 

\color{black}

\begin{figure}
	\includegraphics[width=0.8\textwidth]{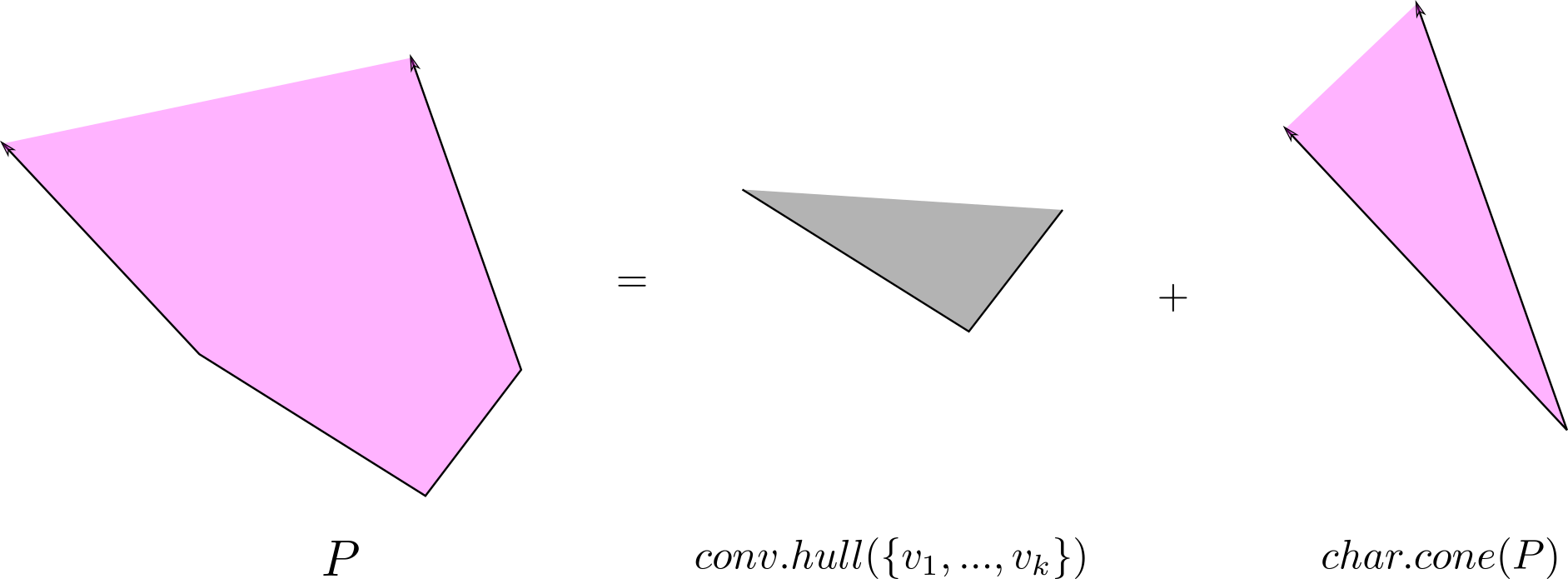}
	\caption{An unbounded, pointed polyhedral set expressed as the Minkowski sum of its characteristic cone and the convex hull of its vertices.}
\end{figure}

\subsubsection{The base map}

For a pointed polyhedral set $P$, we first define a function $\delta_P$ which measures the distance \emph{along directions in $\textrm{char.cone}(P)$} between a point in $P$ and $K_P$.

\begin{lemma} \label{l:basewelldefined}
Let $P \subset \mathbb{R}^n$ be a pointed polyhedral set.  
Define the function $\delta_P:P \to \mathbb{R}^{\geq 0}$ by 
\begin{equation*} \label{eq:D(p)}
\delta_P(p) := \inf \left \{ t \in \mathbb{R}^{\geq 0}\mid p = x+ ty, x \in K_P, y \in \textrm{char.cone}(P), |y| = 1 \right \}.
\end{equation*} 
Then for each point $p \in P$, the infimum in the definition of $\delta_P(p)$ is attained by a unique point in $K$ (i.e. there exists a unique choice of $x \in K$ such that  $p = x+\delta_P(p)y$ for some 
$y \in \textrm{char.cone}(P)$ with $|y| = 1$). 
\end{lemma}

\begin{proof}
We first show that the infimum is attained for every $p \in P$.  By Theorem \ref{t:polyhedraldecomposition}, the set of points $x \in K_P$ such that $p = x+ty$ for some $t \geq 0$ and $y \in \textrm{char.cone}(P)$ is nonempty.
Hence there exist sequences  $\{x_n\}_{n \in \mathbb{N}} \subset K_P$, $\{y_n\}_{n \in \mathbb{N}} \subset \textrm{char.cone}(P)$ with $|y_n| = 1$ for all $n$,  and $\{t_n\}_{n \in \mathbb{N}} \subset \mathbb{R}^{\geq 0}$  with $\lim_{n \to \infty} t_n = \delta_P(p)$ 
such that $p = x_n+t_ny_n.$ 

   Since $K_P$ is compact, we may assume without loss of generality that the sequence $x_n$ converges to a point $x_{\infty} \in K_P$.  It follows that the sequence of vectors $t_ny_n = p-x_n$ also converges. Since $|y_n| = 1$ for all $n$, both the sequences $\{t_n\}$ and $\{y_n\}$ must converge; denote their limits by $t_{\infty}$ and $y_{\infty}$.  Notice $|y_{\infty}| = 1$ and $t_{\infty} = \delta_P(p)$.  Since $y_n \in \textrm{char.cone}(P)$ for all $n$ and $\textrm{char.cone}(P)$ is closed (for any polyhedral set $P$, $\textrm{char.cone}(P)$ is the intersection of finitely many closed half spaces), $y_{\infty}$ is also in $\textrm{char.cone}(P)$.  Thus $p = x_{\infty} + t_{\infty}y_{\infty}$, as desired.

Now, to show uniqueness, suppose there are two distinct points $x_1$ and $x_2$ in $K_P$ with corresponding unit-length vectors $y_1$ and $y_2$ in $\textrm{char.cone}(P)$ such that $$p = x_1+\delta_P(p)y_1 =  x_2+\delta_P(p)y_2.$$  Since $K_P$ is convex, the point 
$$\alpha \coloneqq \frac{x_1+x_2}{2}$$ is in $K$.  But then the line segment connecting $p$ and $\alpha$ is the median of the isosceles triangle with vertices $p$, $x_1$ and $x_2$, and hence  
$$|p-\alpha| < \delta_P(p).$$
Furthermore, 
$$\alpha + \frac{\delta_P(p)y_1 + \delta_P(p)y_2}{2}  = p$$ and 
$$ \frac{\delta_P(p)y_1 + \delta_P(p)y_2}{2} \in \textrm{char.cone}(P)$$
since the characteristic cone is closed under taking nonnegative linear combinations.   
This contradicts the definition of $\delta_P(p)$ as the shortest distance from a point in $K_P$ to $p$ along a vector in $\textrm{char.cone}(P)$. 
\end{proof}

\begin{figure}
	\includegraphics[height=1.5in]{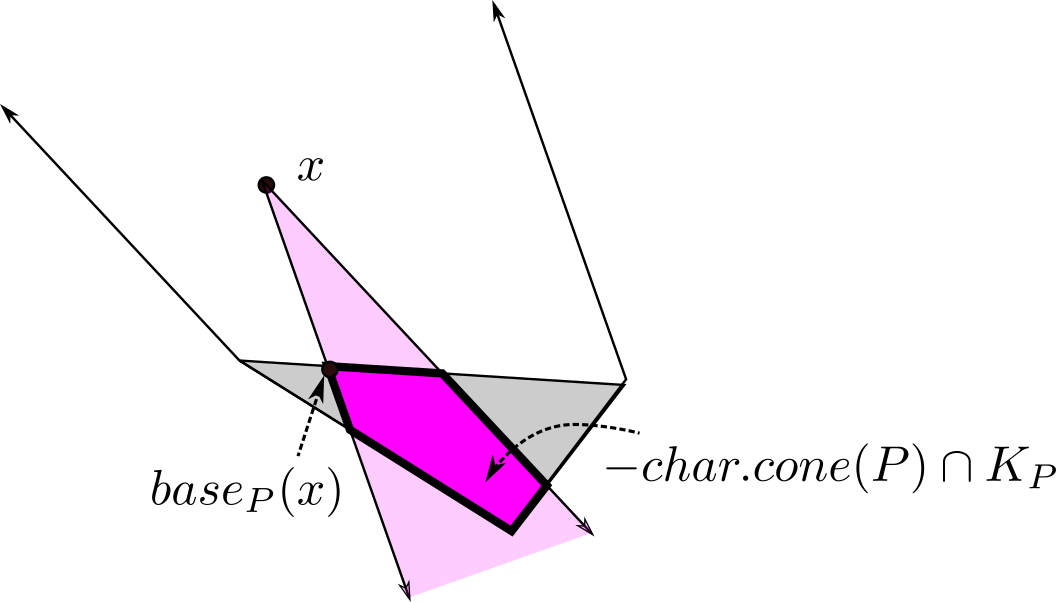}\hspace{1cm}\includegraphics[height=1.5in]{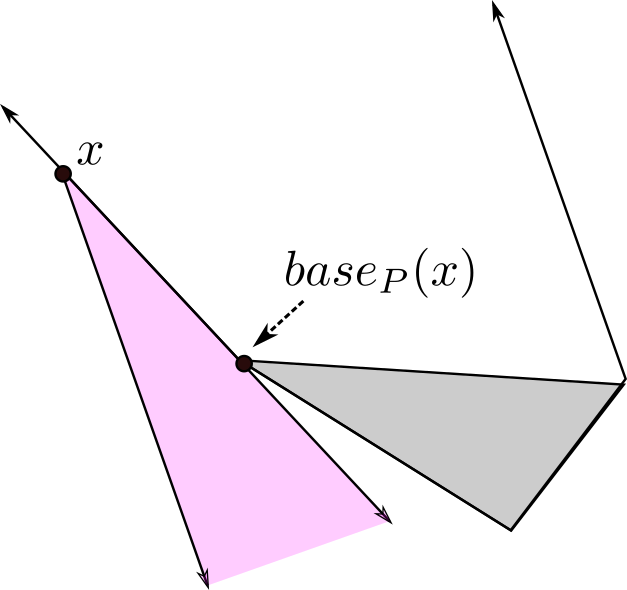}
	
	\caption{The $\textrm{base}_P$ map, visualized as sending each point in $P$ to the closest point in $K_P$ intersected with the negative characteristic cone of $P$ based at $x$.}
\end{figure}

The uniqueness guaranteed by Lemma \ref{l:basewelldefined} enables us to define the \emph{base map} for a single pointed polyhedral set:

\begin{definition}
Let $P \subset \mathbb{R}^n$ be a pointed polyhedral set. 
The \emph{base map} for $P$ is the map 
$$\textrm{base}_P:|P| \to |K_P|$$  defined by setting, for each $p \in |P|$, $\textrm{base}_P(p)$ to be the unique point $x \in |K_P|$ such that 
$$p = x+\delta_P(p)y$$ for some $y \in \textrm{char.cone}(P)$ with $|y|=1$.  
\end{definition}

The following technical lemma will be needed to prove that $\textrm{base}_P$ is a continuous map (Proposition \ref{p:baseInEachCellContinuous}).

\begin{figure}
	\includegraphics[width=3in]{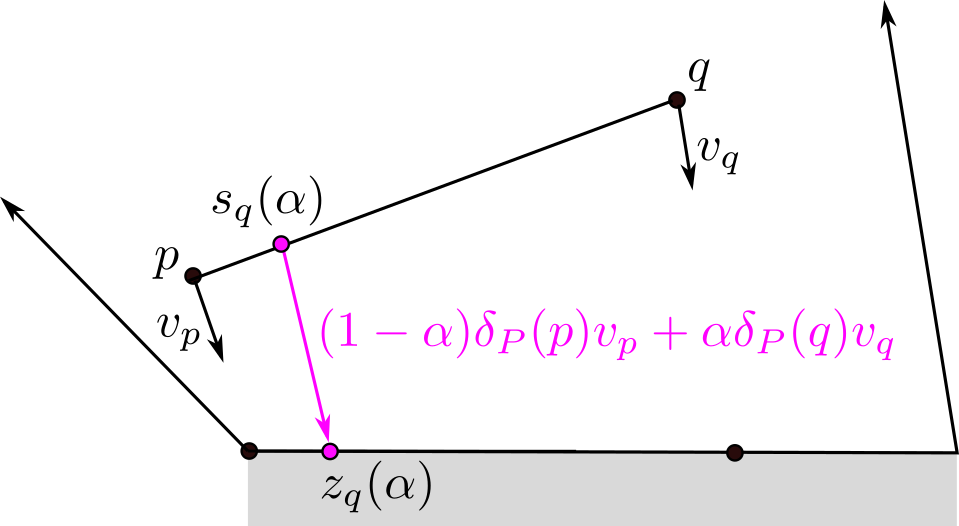}
	\caption{An illustration of the setup for the proof of Lemma \ref{l:deltaPcontinuous}. }
\end{figure}
\begin{lemma} \label{l:deltaPcontinuous}
Let $P \subset \mathbb{R}^n$ be a pointed polyhedral set.  Let $x \in P$.  Then for any $\epsilon >0$, there exists $\eta > 0$ such that for every $x' \in P$,
 if $|x-x'| < \eta$ then $\delta_P(x') \leq \delta_P(x) + \epsilon$. 
\end{lemma}

\begin{proof}
Fix any point $p \in P$.  Denote by $v_p$ the unique element of $-\textrm{char.cone}(P)$ such that $$p+v_p\delta_P(p) = \textrm{base}_P(p).$$ Let $H_1,\ldots,H_n$ be an irredundant collection of hyperplanes that bound $P$.   Then there exists $r = r(p)  > 0$ such that for each $i = 1, \ldots, n$, either $p \in H_i$ or $\overline{B_r(p)} \cap H_i = \emptyset$.  (Here, $\overline{B_r(p)}$ denotes the closed ball of radius $r$ centered at $x$.) Define 
\begin{align*}
D_r(p) := &\{q \in \mathbb{R}^n \mid |q-p| \leq r\} \cap P, \\ 
S_r(p) := &\{q \in \mathbb{R}^n \mid |q-p| = r\} \cap P, \\
M_p := & \sup\{|q - k| \mid q \in D_r(x), k \in K_P\} < \infty.\\
\end{align*} 
For each point $q \in S_r(p)$, define $v_q$ to be the unique element of $-\textrm{char.cone}(P)$ such that $$q + v_q\delta_P(q) = \textrm{base}_P(q).$$
 Also for each $q \in S_r(p)$, define the path $s_q:[0,1] \to P$  by 
$$s_q(\alpha)  \coloneqq (1-\alpha)p +\alpha q.$$  
Note $s_q$ is the straight-line path from $p$ to $q$.  
By construction, 
for any $0 < r' \leq r$, 
\begin{equation} \label{eq:wholenbhd}
\overline{B(p,r')}  \cap P = \bigcup_{q \in S_r(q),  \alpha \in [0, r'/r ]} s_q(\alpha)
\end{equation}

Define, for each $q \in S_r(p)$, the path $z_q:[0,1] \to \mathbb{R}^n$ by 
$$z_q(\alpha) \coloneqq s(\alpha)+ (1-\alpha) \delta_P(p)v_p +  \alpha \delta_P(q)v_q.$$ 
Note that 
$$z_q(\alpha) = (1-\alpha)(p+\delta_P(p)v_p) + \alpha(q+ \delta_P(q)v_q) = (1-\alpha) \textrm{base}_P(p) +\alpha \textrm{base}_P(q).$$
Thus $z_q$ is a straight-line path from $\textrm{base}_P(p)$ to $\textrm{base}_P(q)$. 
Therefore, since $\textrm{base}_P(p)$ and $\textrm{base}_P(q)$ are in $K_P$ and $K_P$ is convex, $z_q(\alpha) \in K_P$ for all $\alpha$.  

Note that 
$$(1-\alpha)\delta_P(p)v_p +\alpha\delta_P(q)v_q \in -\textrm{char.cone}(P)$$ since characteristic cones are, by definition, closed under taking nonnegative linear combinations.  
Thus, for any $q \in S_r(p)$ and any $\alpha \in [0,1]$, 
\begin{multline} \label{eq:badbound}
\delta_P(s_q(\alpha)) \leq |s_q(\alpha) - z_q(\alpha)| =|(1-\alpha) \delta_P(p)v_p + \alpha \delta_P(q)v_q|  \\
 \leq |\delta_P(p)v_p| + |\alpha\delta_P(p)v_p| + |\alpha \delta_P(q)v_q| \leq \delta_P(p) + \alpha 2 M.
\end{multline}

 Set $$\eta \coloneqq \min \left \{ \frac{\epsilon r}{2M},r \right \}.$$  Consider any point $p' \in P$ such that 
 $$|p'-p| < \eta =  \min \left\{ \frac{\epsilon r}{2M},r \right\} $$
   By \eqref{eq:wholenbhd}, there exists a point $q_0 \in S_r(p)$ and a real number $\alpha_0 \in [0, \frac{\epsilon}{2M}]$ such that $p' = s_{q_0}(\alpha_0)$. 
   Hence, by \eqref{eq:badbound}, we have
   $$\delta_P(p') = \delta_P(s_{q_0}(\alpha_0)) \leq \delta_P(p) + \alpha_0 2 M \leq \delta_P(p) + \epsilon.$$

\end{proof}

\begin{proposition} \label{p:baseInEachCellContinuous}
Let $P \subset \mathbb{R}^n$ be a pointed polyhedral set of dimension $n$.  
Then the function $\textrm{base}_P:|P|  \to |K_P|$ is continuous.  
\end{proposition}

\begin{proof} 
Suppose $\textrm{base}_P$ is not continuous at $p \in P$.  This means there exists $\epsilon > 0$ and a sequence of points $\{p_n\}_{n \in \mathbb{N}}$ in $P$ such that 
$\lim_{n \to \infty} p_n = p$ but  $$|\textrm{base}_P(p_n) - \textrm{base}_P(p)| \geq \epsilon$$ for all $n \in \mathbb{N}$.  Since $K_P$ is compact, we may, by passing to a subsequence if necessary, assume without loss of generality that the sequence $\{\textrm{base}_P(p_n)\}_{n \in \mathbb{N}}$ converges to a point $x' \in K_P$.  Notice $\textrm{base}_P(p) \neq x'$, since $|\textrm{base}_P(p)-x'| \geq \epsilon$.

The vectors $p_n - \textrm{base}_P(p_n)$ are all in $\textrm{char.cone}(P)$ by definition.  
Since 
\begin{equation} \label{eq:twolimits}
p-x' =  \left(\lim_{n \to \infty} p_n \right) -  \left( \lim_{n \to \infty} \textrm{base}_P(p_n) \right) = \lim_{n \to \infty} (p_n - \textrm{base}_P(p_n)),
\end{equation}
and $\textrm{char.cone}(P)$ is closed, \eqref{eq:twolimits} implies 
\begin{equation} \label{eq:p-x'}
p-x' 
 \in \textrm{char.cone}(P).
\end{equation}

 Now fix any $\epsilon > 0$. Pick $N_1 \in \mathbb{N}$ such that $|p-p_n| < \epsilon/2$ for all $n \geq N_1$. By Lemma 
 \ref{l:deltaPcontinuous}, we may pick $N_2 \in \mathbb{N}$ such that 
$$\delta_P(p_n) = |p_n - \textrm{base}_P(p_n)| \leq |p-\textrm{base}_P(p)| + \epsilon/2$$
 for all $n \geq N_2$.  Hence, for $n \geq \textrm{max}\{N_1,N_2\}$, we have 
\begin{multline} \label{eq:notfaroff}
|p-\textrm{base}_P(p_n)| \leq |p-p_n| + |p_n - \textrm{base}_P(p_n)|  \\ 
\leq \epsilon/2 + |p - \textrm{base}_P(p)| + \epsilon/2 = \delta_P(p) + \epsilon.
\end{multline}
Since $\epsilon$ was arbitrary, \eqref{eq:notfaroff} implies
\begin{equation} \label{eq:lengthbound} |p-x'| = |p- \lim_{n \to \infty} \textrm{base}_P(p_n)| \leq |p-\textrm{base}_P(p)|. \end{equation}

Thus 
\begin{equation*}  p = x'+ |p-x'| \frac{p-x'}{|p-x'|},\end{equation*}
where $x'$ is a point in $K_P$ distinct from $\textrm{base}_P(p)$, $\frac{p-x'}{|p-x'|} \in \textrm{char.cone}(P)$ by \eqref{eq:p-x'} and has unit length, and $|p-x'| \leq \delta_P(p)$ by \eqref{eq:lengthbound}.  
This contradicts the fact that $x=\textrm{base}_P(p)$ is the unique (by Lemma \ref{l:basewelldefined}) point in $K_P$ such that $p = x+c v$ for some $x \in K_P$, unit vector $v \in \textrm{char.cone}(P)$ and constant $0 \leq c \leq \delta_P(p)$.  

\end{proof}

\begin{lemma} \label{l:retractionpreservesfacets}
Let $P \subset \mathbb{R}^n$ be a pointed polyhedral set and let $F$ be a face of $P$.  Then $\textrm{base}_{P}(F) \subseteq F$.  
\end{lemma}

\begin{proof}  
If $F$ is either $P$ itself or a $0$-cell of $P$, the result is immediate.  So assume $F$ is a proper face of $P$ (i.e. $0 < \textrm{dim}(F) < \textrm{dim}(P)$) and 
fix a point $p \in F$.  Suppose, for a contradiction, there exists $x \in K_P \setminus F$ such that $p = x+ty$ for some $y \in \textrm{char.cone}(P)$ with $|y| = 1$ and $t > 0$. 

Let $H$ be the affine hull of $F$.  Since $P$ is convex, $F = P \cap H$.   
Since $x \not \in F$, this implies $x \not \in H$. Therefore the vector $y = \tfrac{p - x}{t}$ is not in the tangent space $T_pH$, i.e. the ray $\{x+ty \mid t \geq 0\}$ ``crosses'' $F$.    Since $P$ is convex and $F$ is in the boundary of $P$, points of the form $x+ty$ are in $P$ whenever $0 \leq t \leq |p-x|$ and are not in $P$ whenever $t > |p-x|$.  In particular, 
$$x+(|p-x|+\epsilon)y = p+\epsilon y \not \in P$$ for $\epsilon > 0$.   This contradicts the assumption that $y \in \textrm{char.cone}(P)$.
\end{proof}

\begin{lemma} \label{l:facebyface}
Let $P$ be an $n$-dimensional polyhedral set in $\mathbb{R}^n$ and let $F$ be a face of $P$.  Then the maps $\textrm{base}_P$ and $\textrm{base}_F$ agree on $|F|$. 
\end{lemma}

\begin{proof}
Since $\textrm{base}_P(F) \subset F$ by Lemma \ref{l:retractionpreservesfacets} and $F \subset P$, we must have $\delta_{P}(p) = \delta_{F}(p)$ for every point $p \in F$.   By Lemma \ref{l:basewelldefined}, there is a unique point in $F$ that that realizes this minimum.  
\end{proof}

An immediate corollary of Lemma \ref{l:facebyface} is that base maps for different cells of a polyhedral complex in $\mathbb{R}^n$ agree on shared faces:

\begin{corollary} \label{c:globalretractwelldefined}
Let $P_1$ and $P_2$ be two distinct, pointed, polyhedral sets in a polyhedral complex embedded in $\mathbb{R}^n$ that share a common face $F$.  Then for every point $p \in |F|$, $\textrm{base}_{P_1}(p) = \textrm{base}_{P_2}(p)$. 
\end{corollary}

As a consequence of Corollary \ref{c:globalretractwelldefined}, we can formulate the following definition:

\begin{definition} \label{d:basemap}
Let $\mathcal{C}$ be a polyhedral complex embedded in $\mathbb{R}^n$ in which all  cells are pointed.   Define the function 
$$\textrm{base}_{\mathcal{C}}: |\mathcal{C}| \to |\textrm{com.part}(\mathcal{C})|$$
 by whichever of the functions $\textrm{base}_P$, for $P \in \mathcal{C}$, is defined. 
\end{definition}

\begin{proposition} \label{l:basecontinuous}
 For any polyhedral complex $\mathcal{C}$ embedded in  $\mathbb{R}^n$, the map $\textrm{base}_{\mathcal{C}}$ is continuous.  

\end{proposition}

\begin{proof}
Proposition  \ref{p:baseInEachCellContinuous} asserts that for each top-dimensional cell $P \in \mathcal{C}$, the map $\textrm{base}_P$ is continuous.  Corollary \ref{c:globalretractwelldefined} asserts that if two cells share a face, the base maps for those cells agree on the shared face.  Consequently, $\textrm{base}_{\mathcal{C}}$ is continuous. 
\end{proof}

\begin{proposition} \label{p:deformationretract}
Let $\mathcal{C}$ be a polyhedral complex embedded in $\mathbb{R}^n$ in which all cells are pointed.  
Define $\Phi: |\mathcal{C}| \times [0,1] \to |\mathcal{C}|$ by 
$$\Phi(p,t) = t \textrm{base}_{\mathcal{C}}(p) + (1-t)p.$$
Then $\Phi$ is a strong deformation retraction of $|\mathcal{C}|$ onto $|\textrm{com.part}(\mathcal{C})|$.  
\end{proposition}

\begin{proof}
It suffices to show that $\Phi$ is continuous --  but this is an immediate consequence of the fact that $\textrm{base}_{\mathcal{C}}$ is continuous (Proposition \ref{l:basecontinuous}).  
\end{proof}

\begin{proof}[Proof of Proposition ~\ref{p:pointeddeformationretraction}]
The strong deformation retract $\Phi$ of Proposition \ref{p:deformationretract} induces the map $\textrm{base}_{\mathcal{C}}:|\mathcal{C}| \to |\textrm{com.part}(\mathcal{C})|$.  By Lemma \ref{l:retractionpreservesfacets}, the map $\textrm{base}_{\mathcal{C}}$ in turn induces a cellular map
$$\underline{\textrm{base}}_{\mathcal{C}}:\mathcal{C} \to \textrm{com.part}(\mathcal{C})$$ defined by setting $\underline{\textrm{base}}_{\mathcal{C}}(C)$ to be the smallest cell in $\textrm{com.part}(\mathcal{C})$ that contains the set $\textrm{base}_{\mathcal{C}}(C)$.  Since $\textrm{base}_{\mathcal{C}}$ is the identity on $|\textrm{com.part}(\mathcal{C})|$, the cellular map $\underline{\textrm{base}}_{\mathcal{C}}$ is surjective.  
\end{proof}

\subsection{The canonical polytopal complex of unpointed polyhedral complexes} \label{s:essentialization}

This subsection extends the definitions and results of Subsection \ref{s:polytopalcpxpointed} to arbitrary polyhedral complexes $\mathcal{C}$ in $\mathbb{R}^n$ -- dropping the requirement that all cells of $\mathcal{C}$ be pointed.  

To begin, we establish the \emph{Pointedness Dichotomy} (Corollary \ref{c:PolyDecompAllPointedOrUn}): {\em If $\mathcal{C}$ is a connected polyhedral complex imbedded in $\mathbb{R}^n$, then either every cell of $\mathcal{C}$ is pointed or every cell of $\mathcal{C}$ is unpointed}.   This will allow us to associate to {\em any} connected polyhedral complex $\mathcal{C}$ a homotopy-equivalent {\em pointed} connected polyhedral complex, which we call its \emph{essentialization}. We can then apply the results of Subsection \ref{s:polytopalcpxpointed} to the essentialization of $\mathcal{C}$ to obtain a homotopy-equivalent compact model for $\mathcal{C}$.

\subsubsection{Unpointed polyhedral complexes and their essentializations} For a hyperplane arrangement $\mathcal{A}$ in $\mathbb{R}^n$, we will call the linear space spanned by the normal vectors of $\mathcal{A}$ the \emph{normal span} of $\mathcal{A}$ and denote it by $\textrm{n.span}(\mathcal{A})$. The \emph{rank} of a hyperplane arrangement is defined as the dimension of $\textrm{n.span}(\mathcal{A})$.  

 Every polyhedral set $P \subset \mathbb{R}^n$ admits a \emph{irredundant realization} as the intersection of finitely many closed half-spaces, $P =H_1^+ \cap \ldots H_m^+$, such that $P$ is not equal to the intersection of any proper subset of $\{H_1^+,\ldots,H_m^+\}$ (\cite[Sec. 26]{Grunbaum}).  If $P$ is $n$-dimensional, this irredundant realization is unique: the boundaries of these half-spaces are the affine hulls of the $(n-1)$-dimensional faces of $P$.  
 
\begin{definition} \label{d:normalspan}
Let $P \subset \mathbb{R}^n$ be a polyhedral set (of arbitrary dimension).  \begin{enumerate} 
\item Define the \emph{normal span} of $P$, denoted $\textrm{n.span}(P)$, to be the linear space spanned by the set of all vectors $v$ that are normal to some hyperplane $H$ such that $H \cap P = \emptyset$.  
\item Define the \emph{rank} of $P$, denoted $\textrm{rank}(P)$, to be the dimension of $\textrm{n.span}(P)$
\end{enumerate}
\end{definition}
\color{black}

The following well-known lemma asserts that any redundant inequality in a system of linear inequalities with non-empty solution set is a non-negative linear combination of the other linear inequalities in the system:

\begin{lemma}(Farkas' Lemma III,  \cite[Prop. 1.9]{Ziegler})
Let $A \in \mathbb{R}^{m \times d}$, $z \in \mathbb{R}^m$, $a_0 \in \mathbb{R}^d$, $z_0 \in \mathbb{R}$.  
Then $a_0^T x \leq z_0$ is valid for all $x \in \mathbb{R}^d$ with $Ax \leq z$ if and only if at least one of the following hold:
\begin{enumerate}
\item \label{i:Farkasnonempty} there exists a row vector $c \geq \vec{0}$ such that $cA = a_0^T$ and $cz \leq z_0$ 
\item \label{i:Farkasempty} there exists a row vector $c \geq \vec{0}$ such that $cA = \vec{0}$ and $cz < 0$ 
\end{enumerate}
\end{lemma} 

\begin{remark}  All inequalities in Farkas' Lemma III should be interpreted coordinate-wise. Condition \eqref{i:Farkasnonempty} is the statement that the inequality $a_0^T x \leq z_0$ is redundant and can be expressed as a positive linear combination of the other inequalities in the system $Ax \leq z$; condition \eqref{i:Farkasempty} implies that the polyhedral set determined by $Ax \leq z$ is empty.
\end{remark} 

\begin{corollary} \label{c:topdimnspan}
Let $P$ be a polyhedral set of dimension $n$ in $\mathbb{R}^n$.  Then $\textrm{n.span}(P) = \textrm{n.span}(\mathcal{A})$ where $\mathcal{A}$ is the hyperplane arrangement $$\mathcal{A} \coloneqq \{\textrm{aff}(F) \mid F \textrm{ is an }(n-1)\textrm{-cell of } P\}.$$ 
\end{corollary}

\begin{proof} If $F$ is an $(n-1)$-face of $P$, then parallel translating $\textrm{aff}(F)$ along the normal vector to the hyperplane $\textrm{aff}(F)$ in the direction ``away'' from $P$ yields a hyperplane disjoint from $P$.  Hence $\textrm{n.span}(\mathcal{A}) \subset \textrm{n.span}(P)$.  Farkas' Lemma III gives the reverse containment, $\textrm{n.span}(\mathcal{A}) \supset \textrm{n.span}(P)$.
\end{proof}

\begin{lemma} \label{l:linspvsnspan} Let $P$ be a polyhedral set, and let $F$ be a proper face of $P$. Then $\textrm{n.span}(P) = \textrm{n.span}(F)$.
\end{lemma}

\begin{proof} Let $P = H_1^+ \cap \ldots \cap H_k^+$ be a realization of $P$ as an intersection of half-spaces and, for each $i$, let $\vec{v}_i$ be a vector normal to $H_i$. Farkas' Lemma then tells us that $\textrm{n.span}(P) = \mbox{span}\{\vec{v}_1, \ldots, \vec{v}_k\}$.

Note that is immediate from Definition \ref{d:normalspan} that $\textrm{n.span}(P) \subseteq \textrm{n.span}(F)$. To see that $\textrm{n.span}(F) \subseteq \textrm{n.span}(P)$, note that $F = P \cap H'$ for $H'$ a supporting hyperplane of $P$ containing $\textrm{aff}(F)$, and therefore \[F = (H_1^+ \cap \ldots \cap H_k^+) \cap (H')^+ \cap (H')^-\] is a realization of $F$ as an intersection of half-spaces. 

Let $\pm \vec{v}$ be a nonzero vector normal to $H'$ in the direction of $(H')^\pm$. Since $H'$ is a supporting hyperplane of $P$, at least one of $H'\pm \vec{v}$ is disjoint from $P$, so $v \in \textrm{n.span}(P)$.  We then conclude that
 $$\textrm{n.span}(F) = \mbox{span}\{v, v_1, \ldots, v_k\} = \mbox{span}\{v_1, \ldots, v_k\} = \textrm{n.span}(P),$$ as desired.
\end{proof}

\begin{lemma}
Let $\mathcal{C}$ be a polyhedral complex in $\mathbb{R}^n$, and let $P_1, P_2$ be cells of $\mathcal{C}$.   If $P_1$ and $P_2$ share a common face $F$, then $\textrm{n.span}(P_1) = \textrm{n.span}(P_2)$. 
\end{lemma}

\begin{proof}
By Lemma \ref{l:linspvsnspan}, $n.span(P_1) = n.span(F) = n.span(P_2)$.
\end{proof}
The following is now immediate:

\begin{lemma} \label{l:nspanconstant}
Let $\mathcal{C}$ be a connected polyhedral complex in $\mathbb{R}^n$.  Then $\textrm{n.span}(C)$ is independent of $C \in \mathcal{C}$. 
\end{lemma}

In view of Lemma \ref{l:nspanconstant}, we can formulate the following definition:

\begin{definition}
For $\mathcal{C}$ a connected polyhedral complex in $\mathbb{R}^n$, define 
\begin{align*}
\textrm{n.span}(\mathcal{C}) & \coloneqq \textrm{n.span}(C)\\
\textrm{rank}(\mathcal{C}) & \coloneqq \textrm{rank}(C)\\
\end{align*}
 for any cell $C \in \mathcal{C}$.
 \end{definition}
 
 \begin{remark}
The conclusion of Lemma \ref{l:nspanconstant} may not hold if $\mathcal{C}$ is not connected.  For example, consider the polyhedral complex in $\mathbb{R}^2$ consisting of the infinite strip $C_1 \coloneqq [0,1] \times \mathbb{R}$ and the square $C_2 \coloneqq [2,3] \times [2,3]$, along with all the faces of $C_1$ and $C_2$. Then $\textrm{n.span}(C_1) = \mathbb{R} \times 0$ and $\textrm{n.span}(C_2) = \mathbb{R}^2$.  \end{remark}

\begin{lemma} \label{l:npointed}
Let $P \subset \mathbb{R}^n$ be a polyhedral set of dimension $n$. Then $P$ is pointed if and only if $\textrm{rank}(P) = n$. 
\end{lemma}

\begin{proof} 
Without loss of generality, we may assume $P$ contains the origin.  
Denote the polar dual of $P$ by $P^*$:
$$P^*\coloneqq \{z \in \mathbb{R}^n \mid z^Tx \leq 1 \textrm{ for all }x \in P\}.$$
We have
$$n = \textrm{rank}(P) = \textrm{dim}(\textrm{n.span}(P)) = \textrm{dim}(P^*).$$
By Theorem 9.1 of \cite{Schrijver}, $P^*$ is a polyhedron and $P^{**} = P$.  
By Corollary 9.1a of \cite{Schrijver}, $P^{**}$ is pointed. 
\end{proof}

Lemma \ref{l:npointed} has the following immediate corollary:
\begin{corollary}[Pointedness Dichotomy] \label{c:PolyDecompAllPointedOrUn}
Let $\mathcal{C}$ be a connected polyhedral complex in $\mathbb{R}^n$.  If $\textrm{rank}(\mathcal{C}) = n$ then every cell of $\mathcal{C}$ is pointed; otherwise every cell of $\mathcal{C}$ is unpointed. 
\end{corollary}

 \begin{lemma} Let $\mathcal{C}$ be a $n$-dimensional, connected polyhedral complex in $\mathbb{R}^n$.  Let $\mathcal{A}$ be the hyperplane arrangement 
 $$\mathcal{A} \coloneqq \{\textrm{aff}(F) \mid F \textrm{ is an }(n-1)\textrm{-cell of } \mathcal{C}\}.$$ 
 Then $$\textrm{n.span}(\mathcal{A}) = \textrm{n.span}(\mathcal{C}).$$
 \end{lemma}
 
 \begin{proof}
 For any $(n-1)$-cell $F \in \mathcal{C}$, normal vectors to $\textrm{aff}(F)$ are contained in $\textrm{n.span}(F)$. By Lemma \ref{l:nspanconstant}, $ \textrm{n.span}(F) = \textrm{n.span}(\mathcal{C})$.  Hence $\textrm{n.span}(\mathcal{A}) \subset \textrm{n.span}(\mathcal{C})$. 
 
 For the converse, fix any $n$-cell $C \in \mathcal{C}$.  Then Lemma \ref{c:topdimnspan} implies
 $$\textrm{n.span}(\mathcal{C}) = \textrm{n.span}(C) = \textrm{rank}(\mathcal{H}),$$ where $\mathcal{H}$ is the hyperplane arrangement 
$$\mathcal{H} \coloneqq \{\textrm{aff}(D) \mid D \textrm{ is an }(n-1)\textrm{-face of } C\}.$$ 
Since $\mathcal{H} \subset \mathcal{A}$, this yields $\textrm{n.span}(\mathcal{C}) \subset \textrm{n.span}(\mathcal{A})$. 
 \end{proof}

\begin{remark} \label{rem:nonexample}
A decomposition like the one shown in Figure \ref{f:nonexample} is not a counterexample to Corollary \ref{c:PolyDecompAllPointedOrUn}, because it is \emph{not} a polyhedral decomposition of $\mathbb{R}^2$ (c.f. Section \ref{s:polybackground}). Specifically, $C_2 \cap C_3$ is not a face of $C_2$, as it cannot be written as $C_2 \cap H$ for a supporting hyperplane $H$ of $C_2$. 
Similarly, $C_2 \cap C_4$ is not a face of $C_2$.  

\begin{figure}[h!]
\begin{tikzpicture}[scale=1.5]
\tikzstyle{DoubleArrow} = [<->]
\tikzstyle{arrow} = [->,=stealth]
\draw [DoubleArrow] (0,-1) to (0,1);
\draw [DoubleArrow] (-.75,-1) to (-.75,1);
\draw [arrow] (0,0) to (1,0);
\filldraw  (0,0) circle (1pt);
\node at (-1,0) {$C_1$};
\node at (-.35,0) {$C_2$};
\node at (.45, .45) {$C_3$};
\node at (.45, -.45) {$C_4$};
\end{tikzpicture}
	\caption{The ``nonexample'' of polyhedral decomposition of $\mathbb{R}^2$ of Remark \ref{rem:nonexample}.}
	\label{f:nonexample}
\end{figure}
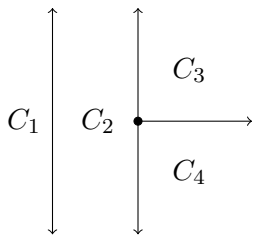
\end{remark}

Recall that the \emph{essentialization} (see e.g. \cite{Stanley}) of a hyperplane arrangement $\mathcal{A}$ in $\mathbb{R}^n$ is the hyperplane arrangement 
$$\textrm{ess}(\mathcal{A}) \coloneqq \{H \cap \textrm{n.span}(\mathcal{A}) \mid H \textrm{ is a hyperplane in } \mathcal{A}\}$$
in the linear space $\textrm{n.span}(\mathcal{A}) \subset \mathbb{R}^n$. 
Denoting by $\text{n.span}(\mathcal{A})^\perp$ the orthogonal complement of $\textrm{n.span}(\mathcal{A})$ in $\mathbb{R}^n$, it follows immediately from the definition that $H' \in \textrm{ess}(\mathcal{A})$ if and only if $H' \oplus \text{n.span}(\mathcal{A})^\perp \in \mathcal{A}$ (\cite{Stanley}).

We define the essentialization of an unpointed polyhedral decomposition of $\mathbb{R}^n$:
\begin{definition}
Let $\mathcal{C}$ be a connected polyhedral complex in $\mathbb{R}^n$.   Define the \emph{essentialization of $\mathcal{C}$} to be the family of sets
  $$\textrm{ess}(\mathcal{C}) \coloneqq \{C \cap \textrm{n.span}(\mathcal{C}) \mid C \in \mathcal{C}\}$$
\end{definition}

\begin{remark} An immediate consequence of Corollary \ref{c:PolyDecompAllPointedOrUn} is that every cell of $\textrm{ess}(\mathcal{C})$ is pointed.  Furthermore,  Corollary \ref{c:PolyDecompAllPointedOrUn} implies that if all cells of $\mathcal{C}$ are pointed, then $\textrm{ess}(\mathcal{C}) = \mathcal{C}$. \end{remark}

The reader may verify the following result:
\begin{lemma} For any connected polyhedral complex $\mathcal{C}$ in $\mathbb{R}^n$, the essentialization of $\mathcal{C}$, $\textrm{ess}(\mathcal{C})$, is a polyhedral complex and $|\textrm{ess}(\mathcal{C})| = \textrm{n.span}(\mathcal{C}) \cap |\mathcal{C}|$. 
\end{lemma}

\begin{lemma} \label{l:essC}
Let $\mathcal{C}$ be a connected polyhedral complex in $\mathbb{R}^n$.
 Then there is a strong deformation retraction $r:|\mathcal{C}| \times [0,1] \to |\textrm{ess}(\mathcal{C})|$ such that 
   \begin{enumerate}
  \item \label{i:essdim} If $C$ is a $k$-cell of $\mathcal{C}$, $r(C)$ is a $(k-n+\textrm{rank}(\mathcal{C}))$-cell of $\textrm{ess}(\mathcal{C})$.
  \item The cellular map $\mathcal{C} \to \textrm{ess}(\mathcal{C})$ given by $C \mapsto r(C)$ is a poset-preserving bijection. 
  \item $r(C) \subset C$ for all $C \in \mathcal{C}$. 
  \end{enumerate}
\end{lemma}

\begin{proof}[Proof sketch] 
Let $\pi:\mathbb{R}^n \to \textrm{n.span}(\mathcal{C})$ denote orthogonal projection onto $\textrm{n.span}(\mathcal{A})$.  Then the linear contraction $r(x,t) \coloneqq xt+\pi(x)(1-t)$ has the desired properties.
\end{proof}

We are now ready to define the canonical polytopal complex associated to a polyhedral complex imbedded in $\mathbb{R}^n$. 

\begin{definition} \label{d:canonicalpolytopalcpx}
Let $\mathcal{C}$ be a connected polyhedral complex in $\mathbb{R}^n$. The \emph{canonical polytopal complex} $\mathcal{K}(\mathcal{C})$ associated to $\mathcal{C}$ is  the polytopal complex
$$\mathcal{K}(\mathcal{C}) \coloneqq \textrm{com.part}(\textrm{ess}(\mathcal{C})).$$ 
\end{definition}

\begin{remark} 
If every cell in $\mathcal{C}$ is pointed, $\textrm{ess}(\mathcal{C}) = \mathcal{C}$, so the canonical polytopal complex $\mathcal{K}(\mathcal{C})$ is just  $\textrm{com.part}(\mathcal{C})$ from Definition \ref{def:compactpart}.
\end{remark}

\begin{theorem} \label{t:strengthenedVersionDeformationRetract}
Let $\mathcal{C}$ be a connected polyhedral complex in $\mathbb{R}^n$.  Then there is a strong deformation retraction $H:|\mathcal{C}| \to |\mathcal{K}(\mathcal{C})|$ that induces a surjective cellular map from $\mathcal{C}$ to the canonical polytopal complex $\mathcal{K}(\mathcal{C})$ that preserves the face poset relation.  
\end{theorem} 

\begin{proof}
 Lemma \ref{l:essC} gives a strong deformation retraction from the arbitrary polyhedral complex $\mathcal{C}$ in $\mathbb{R}^n$ to the pointed polyhedral complex $\textrm{ess}(\mathcal{C})$ that induces a surjective cellular map that preserves the face poset relation.   Proposition \ref{p:pointeddeformationretraction} then applies to the strong deformation retraction of $\textrm{ess}(\mathcal{C})$ onto its compact part.
\end{proof}

\subsection{Homotopy equivalence of sublevel sets} \label{s:homequivsublevel}

\subsubsection{The pointed case}

Recall that the \emph{product complex} $\mathcal{K} \times \mathcal{L}$ of two polyhedral complexes $\mathcal{K}$ and $\mathcal{L}$ is the polyhedral complex consisting of all pairwise products of cells:
$$\mathcal{K} \times \mathcal{L} \coloneqq \{K \times L \mid K \in \mathcal{K}, L \in \mathcal{L}\}.$$  
Definition \ref{d:flevelpreservingisotopy} and Lemma \ref{l:keyGrunertLemma} involve a product complex in which one factor is the complex denoted $[a,b]$, which consists of the real interval $[a,b]$ together with its faces.  

\begin{definition} \cite[Defn. 4.4]{Grunert} \label{d:flevelpreservingisotopy}
Let $M$ be a polyhedral complex and $f:|M| \to \mathbb{R}$ linear on cells.  An \emph{$f$-level-preserving PL isotopy of level sets} between $M_{=a}$ and $M_{=b}$ is a PL homeomorphism 
$$\phi:|M|_{=h} \times [a,b] \to |M|_{\in [a,b]}$$ for $h \in [a,b]$ such that for every $t \in [a,b]$, the restriction of $\phi$ to $|M|_{=h} \times \{t\}$ is a PL homeomorphism between $|M|_{=h} \times \{t\}$ and $|M|_{=t}$. 
\end{definition}

\begin{definition} \label{d:combinatorialequivalence}
Two polyhedral complexes $K$ and $L$ are said to be \emph{combinatorially equivalent} if there is a bijection $\phi:K \to L$ of cells that preserves both dimension and the poset structure given by the face relation. That is:
\begin{enumerate}
\item For all $S, T \in K$, $S$ is a face of $T$ in $K$ if and only if $\phi(S)$ is a face of $\phi(T)$ in $L$ and
\item For all $S \in K$, $\textrm{dim}(S) = \textrm{dim}(\phi(S))$.
\end{enumerate}
\end{definition}

An upshot of Theorem~\ref{t:strengthenedVersionDeformationRetract} is that it enables us to leverage the following result of Grunert, which only applies to polytopal complexes, into a result for polyhedral complexes.  

\begin{lemma}\cite[Lem. 4.13]{Grunert} \label{l:keyGrunertLemma}
Let $M$ be a polytopal complex with a map $f:|M| \to \mathbb{R}$ linear on cells.  Suppose no vertex of $M$ has a value $f(v) \in [a,b]$ under $f$.  Then for any $t \in [a,b]$, the complexes $M_{=t} \times [a,b]$ and $M_{\in [a,b]}$ are combinatorially equivalent
(via the combinatorial equivalence $C_{=t} \times [a,b] \mapsto C_{\in [a,b]}$)
 and there is an $f$-level-preserving PL isotopy $\phi:|M|_{=t} \times [a,b] \to |M|_{\in [a,b]}$. 
\end{lemma}
\color{black}

We will use the following well-known result:

\begin{lemma}[\cite{GL}] \label{l:attainsmaxon0cells}
Let $F:\mathbb{R}^{n_0} \to \mathbb{R}$ be a neural network, and $\mathcal{C}(F)$ be its canonical polyhedral complex. If $\mathcal{P}$ is a cell of $\mathcal{C}(F)$ with at least one vertex as a face, and $F$ achieves a maximum (resp., minimum) on $\mathcal{P}$, then $F$ achieves a maximum (resp., minimum) at a vertex of $\mathcal{P}$.
\end{lemma}

\begin{lemma} \label{l:epsilonNbhd}
Let $\mathcal{C}$ be a polyhedral complex in $\mathbb{R}^n$ of which every cell is pointed.   Let $F:|\mathcal{C}| \to \mathbb{R}$ be linear on cells of $\mathcal{C}$.  Let $[a,b]$ be an interval of transversal thresholds for $F$ on $\mathcal{C}$.  
Then there exists $\epsilon > 0$ such that for every cell $C \in \mathcal{C}$ for which $F(C) \cap [a,b]$ is nonempty, we have
$$F(C) \supseteq [a-\epsilon,b+\epsilon].$$ 
\end{lemma}

\begin{proof}
Since $\mathcal{C}$ has only finitely many cells and $F$, it suffices to show the result for each cell.  
So fix any cell $C\in \mathcal{C}$ such that $F(C) \cap [a,b] \neq \emptyset$.  By continuity of $F$, it suffices to show there exist points $w, v \in C$ such that $F(w) < a$ and $F(v) > b$.  
 
 If $F$ is unbounded above, such a point $v$ clearly exists.  So suppose $F$ is bounded above (on $C$).  Since $F$ is linear on $C$, $F$ achieves its maximum, which we denote $m$.  By assumption, $F$ attains a value $\geq a$ somewhere on $C$, so $m \geq a$. By Lemma \ref{l:attainsmaxon0cells}, $F$ attains $m$ on a $0$-cell of $C$.  Hence $m$ is not a transversal threshold.  Thus $m > b$.  
 
 A similar argument shows the existence of $w \in C$ such that $F(w) < a$. 
\end{proof}

Recall that for a real-valued function $F$, a map $b$ is said to be \emph{$F$-level preserving} if $F(x) = F(b(x))$ for all $x$. 

\begin{lemma} \label{l:FLevelPresDefRetract}
Let $\mathcal{C}$ be a polyhedral complex in $\mathbb{R}^n$ in which all cells are pointed.   Let $F:|\mathcal{C}| \to \mathbb{R}$ be linear on cells.  Let $[a,b]$ be an interval of transversal thresholds for $F$ on $\mathcal{C}$.  Fix $\epsilon > 0$ as in Lemma \ref{l:epsilonNbhd}. Let $I_{a-\epsilon,b+\epsilon}$ be the polyhedral decomposition of $\mathbb{R}$ whose $1$-cells are $(-\infty,a-\epsilon], [a-\epsilon,b+\epsilon]$ and $[b+\epsilon,\infty)$.  
Let $\mathcal{C}'$ be the refinement of $\mathcal{C}$ obtained as the following level set complex:
$$\mathcal{C}' = \mathcal{C}_{F \in I_{a-\epsilon,b+\epsilon}}.$$
Then the following hold:
\begin{enumerate}
\item \label{i:strongdef} The surjective map $\textrm{base}_{\mathcal{C}'}:|\mathcal{C}'| \to |\textrm{com.part}(\mathcal{C}')|$ is induced by a strong deformation retraction.
\item \label{i:stripFlevelPres} The restriction of $\textrm{base}_{\mathcal{C}'}$ to the set 
$$F_{[a,b]} \coloneqq \{x \in |\mathcal{C}| : a \leq F(x) \leq b \}$$ is $F$-level preserving.  Equivalently, for any $a\leq c \leq b$,  
$$\textrm{base}_{\mathcal{C}'}(F_{=c}) = |\textrm{com.part}(\mathcal{C}')_{F = c}|.$$
\item \label{i:trichotomy}  
 \begin{align*}
 \textrm{base}_{\mathcal{C}'}(F_{\leq a-\epsilon}) &= |\textrm{com.part}(\mathcal{C}')_{F \leq a-\epsilon}|, \\ 
 \textrm{base}_{\mathcal{C}'}(F_{\geqq b+\epsilon}) &= |\textrm{com.part}(\mathcal{C}')_{F \geq b+\epsilon}|,\\
  \textrm{base}_{\mathcal{C}'}(F_{[a-\epsilon,a)}) &\subseteq  |\textrm{com.part}(\mathcal{C}')_{F \in [a-\epsilon,a)}|, \\
    \textrm{base}_{\mathcal{C}'}(F_{(b,b+\epsilon]}) &\subseteq  |\textrm{com.part}(\mathcal{C}')_{F \in (b,b+\epsilon]}|, \\
  \textrm{base}_{\mathcal{C}'}(F_{\leq a}) &= |\textrm{com.part}(\mathcal{C}')_{F\leq a}|, \\
  \textrm{base}_{\mathcal{C}'}(F_{\geq b}) &= |\textrm{com.part}(\mathcal{C}')_{F\geq b}|. \\
  \end{align*}
  \item \label{i:subsuplevel} For any $a \leq c \leq b$, 
  \begin{align*}
 \textrm{base}_{\mathcal{C}'} (F_{=c}) & = |\textrm{com.part}(\mathcal{C}')_{F =c}|, \\
  \textrm{base}_{\mathcal{C}'} (F_{\leq c}) & = |\textrm{com.part}(\mathcal{C}')_{F \leq c}|,\\
 \textrm{base}_{\mathcal{C}'} (F_{\geq c}) & = |\textrm{com.part}(\mathcal{C}')_{F \geq c}|.\\
\end{align*}
\end{enumerate}
\end{lemma}

\begin{proof} 

Note that \eqref{i:strongdef} holds by Proposition \ref{p:pointeddeformationretraction}.

 Consider any cell $C' \in \mathcal{C}'$ that has nonempty intersection with $F_{[a,b]}$.  By our choice of $\epsilon$, we must have $F(C') = [a-\epsilon,b+\epsilon]$.  By construction, there exists a cell $C \in \mathcal{C}$ such that 
$$C' = C \cap F_{[a-\epsilon,b+\epsilon]}.$$
So $F(C) \supseteq [a-\epsilon,b+\epsilon]$.
Thus, the level sets $F_{=a-\epsilon}$ and $F_{=b-\epsilon}$ both have nonempty intersection with $C$.  Furthermore, the sets
$F_{=a-\epsilon} \cap C$ and $F_{=b-\epsilon} \cap C$ are faces of $C'$. 

We claim that $\textrm{char.cone}(C')$ consists of directions along with $F$ is constant.  Indeed, consider any point $p \in C'$ and any vector $v \in \textrm{char.cone}(C')$.  Then the point $p+tv \in C'$ for all $t \geq 0$.  By affine-linearity of $F$, $F(p + tv)= F(p) + t\frac{dF}{dv}$.  Thus, if the derivative $\frac{dF}{dv} < 0$, $C'$ must contain a point where $F$ attains a value less than $a-\epsilon /2$, which is impossible.  Similarly, if $\frac{dF}{dv} > 0$,  $C'$ must contain a point where $F$ attains a value greater than $a- \tfrac{\epsilon}{2}$, which is also impossible. Hence $\frac{dF}{dv} = 0$, and the claim follows. 

The map $\textrm{base}_{C'}$ acts on $C'$ by collapsing rays in $\textrm{char.cone}(C')$.  Since $F$ is constant along each ray in $\textrm{char.cone}(C')$, it follows that $\textrm{base}_{C'}$ is $F$-level preserving. 
Since $\textrm{base}_{\mathcal{C}'}$ is defined cell-by-cell and agrees on shared faces, this implies that $\textrm{base}_{\mathcal{C}'}$ is $F$-level preserving on $F_{[a,b]}$, establishing \eqref{i:stripFlevelPres}.

Next, note that by construction of $\mathcal{C}'$, each cell $D \in \mathcal{C}'$ satisfies at least one of the following:
\begin{enumerate}
\item $F(D) \subseteq (-\infty,a-\epsilon]$, 
\item \label{i:aeps} $F(D) \subseteq [a-\epsilon,a)$
\item $F(D) \subseteq [b+\epsilon,\infty)$.
\item  $F(D) \subseteq (b,b+\epsilon]$,
\item \label{i:specialtype} $F(D) = [a-\epsilon,b+\epsilon]$.
\end{enumerate}
The argument above shows that $\textrm{base}_{\mathcal{C}'}$ is $F$-level preserving on cells $D$ of type \eqref{i:specialtype}. The map $\textrm{base}_{\mathcal{C}'}$ is not necessarily $F$-level preserving on cells of the other types, but we do have the containment  $\textrm{base}_{\mathcal{C}'}(D) \subseteq D$ for every cell $D$. Hence:
 \begin{align*}
 \textrm{base}_{\mathcal{C}'}(F_{\leq a-\epsilon}) &= |\textrm{com.part}(\mathcal{C}')_{F \leq a-\epsilon}|, \\ 
 \textrm{base}_{\mathcal{C}'}(F_{\geqq b+\epsilon}) &= |\textrm{com.part}(\mathcal{C}')_{F \geq b+\epsilon}|,\\
  \textrm{base}_{\mathcal{C}'}(F_{[a-\epsilon,a)}) &\subseteq  |\textrm{com.part}(\mathcal{C}')_{F \in [a-\epsilon,a)}|, \\
    \textrm{base}_{\mathcal{C}'}(F_{(b,b+\epsilon]}) &\subseteq  |\textrm{com.part}(\mathcal{C}')_{F \in (b,b+\epsilon]}|.
  \end{align*}
  
Consequently, for any point $p$ in the closed set $ |\textrm{com.part}(\mathcal{C}')_{F \in [a-\epsilon,a]}|$, if $x \in |\mathcal{C}'|$ satisfies $\textrm{base}_{\mathcal{C}}(x) = p$, then at least one of the following holds:
 $x$ belongs to a cell of type \eqref{i:aeps}; $x$ belongs to a cell of type \eqref{i:specialtype} and $F(x) = F(p)$.  Since $\textrm{base}_{\mathcal{C}'}: |\mathcal{C}'| \to |\textrm{com.part}(\mathcal{C}')|$ is surjective, it follows that 
 $$\textrm{base}_{\mathcal{C}'} (F_{\leq a}) = |\textrm{com.part}(\mathcal{C}')_{F \leq a}|.$$
 An analogous argument yields 
  $$\textrm{base}_{\mathcal{C}'} (F_{\geq b}) = |\textrm{com.part}(\mathcal{C}')_{F \geq b}|,$$
establishing \eqref{i:trichotomy}. 
  
  Conclusion \eqref{i:subsuplevel} follows immediately from \eqref{i:stripFlevelPres} and \eqref{i:trichotomy} together.

\end{proof}

\begin{lemma} \label{l:transThreshPersist}
Let $\mathcal{C}, F, a, b, \epsilon, \mathcal{C}'$ be as in Lemma \ref{l:FLevelPresDefRetract}.  Then $[a,b]$ is an interval of transversal thresholds for $F$ with respect to the complex $\textrm{com.part}(\mathcal{C}')$. 
\end{lemma}

\begin{proof}
First, note that $[a,b]$ is an interval of transversal thresholds for $F$ with respect to the complex $\mathcal{C}'$.  This is because the complexes $\mathcal{C}$ and $\mathcal{C}'$ only differ away from points of $F_{[a,b]}$.  Thus, $\mathcal{C}'$ has no flat cells where $F$ takes values in $[a,b]$.  Since $\mathcal{C}'$ and $\textrm{com.part}(\mathcal{C}')$ have the same set of $0$-cells, $F$ does not take a value in $[a,b]$ on any $0$-cell of $\textrm{com.part}(\mathcal{C}')$.  Because  $\textrm{com.part}(\mathcal{C}')$ is a polytopal complex and any nontransversal threshold for a polytopal complex is attained on a $0$-cell of that polytopal complex, it follows that all thresholds in $[a,b]$ are transversal for $F$ with respect to $\textrm{com.part}(\mathcal{C}')$. 
\end{proof}

\begin{proposition} \label{p:PtdHtpyEquiv}
Let $\mathcal{C}$ be a polyhedral complex in $\mathbb{R}^n$ in which all cells are pointed.   Let $F:|\mathcal{C}| \to \mathbb{R}$ be linear on cells of $\mathcal{C}$.  Let $[a,b]$ be an interval of transversal thresholds for $F$ on $\mathcal{C}$. 

Then there exists a polytopal complex $\mathcal{D}$ with $|\mathcal{D}| \subset |\mathcal{C}|$ and a strong deformation retraction $\Phi:|\mathcal{C}| \to |\mathcal{D}|$ (that induces a cellular map from $\mathcal{C}$ to $\mathcal{D}$) such that
$\Phi$ is $F$-level preserving on $|\mathcal{C}_{F \in [a,b]}|$ and for every $c \in [a,b]$,
\begin{equation} \label{eq:PhiProperties}
\Phi( |\mathcal{C}_{F \leq c}| ) = |\mathcal{D}_{F \leq c}|, \quad \Phi( |\mathcal{C}_{F \geq c}| ) = |\mathcal{D}_{F \geq c}|.
\end{equation}
Furthermore, there is an $F$-level preserving PL isotopy  
$$\phi: |\mathcal{D}|_{F = c} \times [a,b] \to |\mathcal{D}_{F \in [a,b]}|.$$
\end{proposition}

\begin{proof}
Let $\mathcal{C}'$ be the refinement of $\mathcal{C}$ constructed in Lemma \ref{l:FLevelPresDefRetract}; set $\mathcal{D} = \textrm{com.part}(\mathcal{C}')$ and $\Phi = \textrm{base}_{\mathcal{C}'}$.  Lemma \ref{l:FLevelPresDefRetract} implies $\Phi$ has the stated properties. 
By Lemma \ref{l:transThreshPersist}, $[a,b]$ is an interval of transversal thresholds for $F$ with respect to the polytopal complex $\mathcal{D}$.  Grunert's result (Lemma \ref{l:keyGrunertLemma}) then gives, for any $c \in [a,b]$, the $F$-level preserving PL isotopy $\phi$. 
\end{proof}
\color{black}

\subsubsection{The unpointed case}

\begin{lemma} \label{l:UnptdToPtd}
Let $\mathcal{C}$ be a connected polyhedral complex in $\mathbb{R}^n$ in which all cells are unpointed, let $F:|\mathcal{C}| \to \mathbb{R}$ be linear on the cells of $\mathcal{C}$, and let $[a,b]$ be a closed interval of transversal thresholds for $F$ on $\mathcal{C}$. 
Then there exists a refinement $\mathcal{C}'$ of $\mathcal{C}$ such that 
\begin{enumerate}
\item the projection map $\textrm{proj}:\mathcal{C}' \rightarrow \textrm{ess}(\mathcal{C}')$ is $F$--level preserving and hence induces a canonical map $F_{\tiny \textrm{ess}}: \textrm{ess}(\mathcal{C}') \rightarrow \mathbb{R}$ which is linear on cells, 
\item  $[a,b]$ is an interval of transversal thresholds for $F_{\tiny \textrm{ess}}$ with respect to $\textrm{ess}(\mathcal{C}')$ 
\end{enumerate}
\end{lemma}

We will find the following notation convenient. Let $c_1, \ldots,  c_k \in \mathbb{R}$ be a set of finitely many distinct real numbers, and let $\mathbb{R}_{c_1, \ldots, c_k}$  be the unique induced polyhedral decomposition of $\mathbb{R}$ whose $0$--cells are $c_1, \ldots, c_k$ and whose $1$--cells are the $k+1$ closed intervals in $\mathbb{R}$ satisfying the condition that they contain at least one $c_i$ in their boundary and no $c_i$ in their interior. Then for a polyhedral complex $\mathcal{C}$ in $\mathbb{R}^n$ and a map $F: \mathcal{C} \rightarrow \mathbb{R}$ that is linear on cells, we will denote by $\mathcal{C}_{c_1 , \ldots , c_k}$ the level set complex associated to $F$ and $\mathbb{R}_{c_1 , \ldots , c_k}$.

\begin{proof} The Pointedness Dichotomy (Corollary \ref{c:PolyDecompAllPointedOrUn}) tells us that since all cells of $\mathcal{C}$ are unpointed, $\mbox{rank}(\mathcal{C}) < n$. 

We begin by noting that since $[a,b]$ is a closed interval of transversal thresholds, and the set of transversal thresholds is open in $\mathbb{R}$, there exists some $\epsilon > 0$ such that $(a-\epsilon, b + \epsilon)$ is an open interval of transversal thresholds. Moreover, $\mbox{rank}(\mathcal{C}_c) = \mbox{rank}(\mathcal{C}_{c'})$ for all $c, c' \in (a-\epsilon, b+\epsilon)$. 

The complex $\mathcal{C}'$ can now be constructed from $\mathcal{C}$ by taking finitely many refinements as follows. If there exists $c_1 \in \mathbb{R}$ for which $\mbox{rank}(\mathcal{C}_{c_1}) > \mbox{rank}(\mathcal{C})$, then there exists a threshold $c_1' \in (\mathbb{R} \setminus [a,b])$ for which 
$$\mbox{rank}(\mathcal{C}_{c_1'}) = \mbox{rank}(\mathcal{C}_{c_1}) > \mbox{rank}(\mathcal{C}).$$ If $\mbox{rank}(\mathcal{C}_{c_1'}) = n,$ then the Pointedness Dichotomy tells us that $\mathcal{C}_{c_1'}= \textrm{ess}(\mathcal{C}')$, and hence the induced map on $\mathcal{C}' :=  \mathcal{C}_{c_1'} = \textrm{ess}(\mathcal{C}')$ is trivially $F$--level preserving. Since, by construction, we have introduced no new cells mapping into the interval $[a,b]$, it remains an interval of transversal thresholds.

If $\mbox{rank}(\mathcal{C}_{c_1'})<n$, repeat the process in the paragraph above until we have found $c_1', \ldots, c_k' \in (\mathbb{R} \setminus [a,b])$ for which $\mathcal{C}_{c_1',\ldots , c_k'}$ has rank $m \leq n$ and there exists no $d \in \mathbb{R}$ for which $\mbox{rank}(\mathcal{C}_{c_1', \ldots, c_k',d}) > m$. Note that since each step in the process increases rank, we know that $k \leq (n- \mbox{rank}(\mathcal{C}))$ is finite.

We now claim that the finite refinement $\mathcal{C}' := \mathcal{C}_{c_1', \ldots, c_k'}$ satisfies the required conditions. If $m = n$, the argument is as above. If $m < n$, then $\mbox{rank}(\mathcal{C}'_d) = \mbox{rank}(\mathcal{C}')$ for every $d \in \mathbb{R}$. But this implies that every level set of $F$ is invariant under translation by vectors in $\mbox{lin.space}(\mathcal{C}')$, and hence the projection map $proj: \mathcal{C}' \rightarrow \textrm{ess}(\mathcal{C}')$ is $F$--level preserving. Since \color{purple} $proj$ \color{black} is bijective and $F$--level preserving on cells, the interval $[a,b]$ remains an interval of transversal thresholds for $F_{\textrm{ess}}: \textrm{ess}(\mathcal{C}') \rightarrow \mathbb{R}$. 
\end{proof}

We are now ready to prove Theorems \ref{t:retraction} and \ref{t:homotopyequiv}, which we restate below for convenience.

\begin{retractiontheorem} 
Let $\mathcal{C}$ be a polyhedral complex in $\mathbb{R}^n$, let $F:|\mathcal{C}| \to \mathbb{R}$ be linear on cells of $\mathcal{C}$, and  let $[a,b]$ be an interval of transversal thresholds for $F$ on $\mathcal{C}$.  Then there exists a polytopal complex $\mathcal{D}$ with $|\mathcal{D}| \subset |\mathcal{C}|$ and a strong deformation retraction $\Phi:|\mathcal{C}| \to |\mathcal{D}|$ (which induces a cellular map $\mathcal{C} \to \mathcal{D}$) such that
$\Phi$ is $F$-level preserving on $|\mathcal{C}_{F \in [a,b]}|$ and for every $c \in [a,b]$,
\begin{equation} \label{eq:PhiProperties}
\Phi( |\mathcal{C}_{F \leq c}| ) = |\mathcal{D}_{F \leq c}|, \quad \Phi( |\mathcal{C}_{F \geq c}| ) = |\mathcal{D}_{F \geq c}|.
\end{equation}
Furthermore, there is an $F$-level preserving PL isotopy  
$$\phi: |\mathcal{D}|_{F = c} \times [a,b] \to |\mathcal{D}_{F \in [a,b]}|.$$
\end{retractiontheorem}

\begin{proof}
It suffices to prove the result for each connected component of $\mathcal{C}$.  Thus, without loss of generality, assume $\mathcal{C}$ is connected. If $\mathcal{C}$ is pointed, the result is Proposition \ref{p:PtdHtpyEquiv}.  If $\mathcal{C}$ is unpointed, Lemma  \ref{l:UnptdToPtd} yields that the projection map $\textrm{proj}$ is a strong deformation retraction that is $F$-level preserving from $\mathcal{C}$ to the pointed polyhedral complex $\mathcal{D} \coloneqq \textrm{ess}(\mathcal{C}')$; it furthermore guarantees that $[a,b]$ is an interval of transversal thresholds for $F$ with respect to $\mathcal{D} = \textrm{ess}(\mathcal{C}')$.  We may therefore apply Proposition  \ref{p:PtdHtpyEquiv} to the action of $F$ on $\textrm{ess}(\mathcal{C}')$; let $\Phi$ be the composition of $\textrm{proj}$ and the strong deformation retraction guaranteed by Proposition  \ref{p:PtdHtpyEquiv}. 
\end{proof}

\begin{homotopyequivtheorem}
Let $\mathcal{C}$ be a polyhedral complex in $\mathbb{R}^n$, let $F:|\mathcal{C}| \to \mathbb{R}$ be linear on cells of $\mathcal{C}$, and let $[a,b]$ be an interval of transversal thresholds for $F$ on $\mathcal{C}$.  Then for any threshold $c \in [a,b]$, $|\mathcal{C}_{F \leq c}|$ is homotopy equivalent to $|\mathcal{C}_{F \leq a}|$; also $|\mathcal{C}_{F \geq c}|$ is homotopy equivalent to $|\mathcal{C}_{F \geq b}|$. \end{homotopyequivtheorem}

\begin{proof}
By Theorem \ref{t:retraction}, $|\mathcal{C}_{F \leq c}|$ is homotopy equivalent to $|\mathcal{D}_{F \leq c}|$ and $|\mathcal{C}_{F \leq a}|$ is homotopy equivalent to $|\mathcal{C}_{F \leq a}|$, where $\mathcal{D}$ is as constructed in that theorem.  It furthermore gives that  $|\mathcal{D}_{F \leq c}|$ is homotopy equivalent to $|\mathcal{D}_{F \leq a}|$.  Since homotopy equivalence is a transitive relation, this implies $|\mathcal{C}_{F \leq c}|$ is homotopy equivalent to $|\mathcal{C}_{F \leq a}|$.  The proof for the superlevel sets is the same. 

\end{proof}

\section{Local $H$--complexity can be arbitrarily large} \label{sec:LocCmpxArbLarge}

\color{black}
In this section, we present examples illustrating the phenomenon that the local $H$--complexity of a ReLU neural network can be arbitrarily large.  

We begin by establishing a few constraints on the $\nabla F$-orientations of the $1$--complex of $\mathcal{C}(F)$.
\color{black}
\begin{lemma}
	Let $F:\mathbb{R}^n \to \mathbb{R}$ be a ReLU neural network. Suppose $C$ is a flat, top-dimensional cell of $\mathcal{C}(F)$ and let $C'$ be a top-dimensional cell of $\mathcal{C}(F)$ such that $C \cap C'$ is a mutual facet $D$. Then any pair $E_1, E_2$ of edges of $C'$ with one vertex in $C$ satisfy the condition that $E_1$ and $E_2$ have the same induced orientation by $F$ (that is, they are both oriented outwards or both oriented inwards).  
\end{lemma}

\begin{proof}
	Let $x_1$ and $x_2$ be points on $E_1$ and $E_2$ which are not vertices of $\mathcal{C}(F)$. Now, $x_1, x_2$ are points in $C'$ which are not in $C$. Suppose $F(x_1)>F(C)$. Then we must also have $F(x_2)>F(C)$. Supposing otherwise leads to a contradiction: If $F(x_2)=F(C)$ then $C'$ has a facet with $F(C')=F(C)$ and a point $x_2$ with $F(x_2)=F(C)$, so as $F$ is affine, $F(C')$ must be constant. This contradicts the assumption that $F(x_1)>F(C)$. Similarly, if $F(x_2)<F(C)$ then by the intermediate value theorem there is a point $x$ on the line segment between $x_1$ and $x_2$ satisfying $F(x)=F(C)$. As $C'$ is convex, $x \in C'$. As $F(C')$ is constant on $D$ and on a point $x$ not in $D$ then $F$ must be constant on $C'$, again a contradiction. This means that if $F(x_1)>F(C)$ then $F(x_2)>F(C)$ as well, and $E_1$ and $E_2$ have the same induced orientation.The same argument applies if $F(x_1)<F(C)$. 
	
	Lastly, if $F(x_1)=F(C)$, as seen before, $C'$ must also be a flat cell, and so both edges are assigned no orientation. \\
\end{proof}

\begin{lemma}
	
	Suppose $F_1: \mathbb{R}^n \to \mathbb R^m$, is a generic layer map which realizes an $n$-dimensional flat cell $C$. Let $\mathcal{E}$ be the set of equivalence classes of edges of $\mathcal{C}$ with one vertex in $C$, under the equivalence that $[E_1] = [E_2]$ if and only if $E_1$ and $E_2$ are both edges of the same $n$-cell $C'$ with $D=C'\cap C$ being a mutual facet. \\ 
	Let $s: \mathcal{E}\to \{-1,1\}$ be any function. Then it is possible to select an affine map $G_s: \mathbb{R}^m \to \mathbb{R}$ in such a way that if $E$ is an edge in $\mathcal{C}$ with one vertex in $C$, the neural network $F_s = G_s \circ F_1 $ satisfies the condition that $F_s(E\backslash C)>0$ if and only if $s([E])=1$ and $F_s(E\backslash C)<0$ if and only if $s([E])=-1$. 
	
\end{lemma}

\begin{proof}
	If $F: \mathbb{R}^n \to \mathbb{R}^m$ is a generic layer map with a top-dimensional flat cell, we have established in Lemma \ref{l:onelayersendstopoint} that $C =   H_1^- \cap ... \cap H_m^-$. Without loss of generality, we treat $C$ as if it has $m$ facets given by $H_i \cap C$; if it has fewer than $m$ facets, and if hyperplane $H_i$ is a hyperplane in the arrangement $\mathcal{A}_1$ associated to $F_1$ which is not a supporting hyperplane of $C$, then all cells in the star of $C$ are contained in $H_i^-$ and are sent to zero in coordinate $i$ under $F_1$. We could therefore construct a layer $F_1'$ with all output coordinates of the affine map $A_1$ corresponding to nonsupporting hyperplanes of $C$ set equal to zero, by setting their weights equal to zero. This new layer is equal to $F_1$ on all points in cells in the star of $C$. If $G: \mathbb{R}^m  \to \mathbb{R}$ is an affine second layer, its restriction to the nonzero input coordinates completely determines the induced orientations on edges incident to $C$ in both $F_1$ and $F_1'$, so we may define $G'$ on this subspace and choose an arbitrary affine extension of $G'$ to $\R^m$ for the result. We note that by deleting the zero coordinates of $\R^m$ we obtain a smaller network with the dimension of the hidden layer equal to the number of facets of $C$ with equivalent orientations on edges incident to $C$.

	Now under these assumptions, $C= H_1^- \cap ... \cap H_m^-$ is sent to the origin by $F_1$ and it has $m$ facets given by $H_i \cap C$ for $i \in \{1,...,m\}$. 
	
	If $m< n$ then we are done as there there are no vertices and therefore no edges in $\mathcal{C}(F)$ which have a vertex in $C$ to be assigned orientations, so consider the case that $m \geq n$.
	The equivalence classes used to define $\mathcal{E}$ are well-defined. Letting $\vec{r_i}$ be a length-$m$ string with a single $1$ in position $i$ and $-1$ elsewhere, we see immediately $R_{\vec{r_i}}$, which is a nonempty $n$-cell given by $H_1^- \cap ... \cap H_i^+ \cap ... \cap H_m^-$, shares a facet with $C$ given by $H_i \cap C$. Additionally, every region which shares a facet with $C$ is equal to $R_{\vec{r_i}}$ for some $i$ as the facets of $C$ are precisely $H_i \cap C$. No edge representing an element of $\mathcal{E}$ belongs to both $R_{\vec{r_i}}$ and $R_{\vec{r_j}}$ for $i \neq j$ because any point in $R_{\vec{r_i}}\cap R_{\vec{r_j}}$  must be in $R_{\vec{-1}}$. 
	
Furthermore, every edge $E$ incident to $C$ is an edge of  $R_{\vec{r_i}}$ for some $i$. Suppose $E$ is incident to $C$. Then it has a vertex $v$ given by $E \cap C$. By the genericity of $A_1$, $v$ is contained in $n$ of the hyperplanes $H_1,...H_m$. To be a vertex of $C$, it must additionally have a $-1$ ternary coding relative to the $m-n$ remaining hyperplanes. As the remaining $m-n$ hyperplanes do not meet $v$, $E$ also has a $-1$ labeling with respect to the remaining $m-n$ hyperplanes, and is contained in $n-1$ of the $n$ hyperplanes which intersect to form $v$. Let $H_i$ be the hyperplane containing $v$ but not $E$. If $E$ was contained in $H_i^-$, then $E$ would be an edge of $C$ and not incident to $C$. So, $E$ is contained in $H_i^+$. So, we may identify $E$ as being an edge contained in the closure of $H_i^+ \cap \left(\bigcap_{j \neq i}H_j^-\right)$, and so $E$ is an edge of $R_{\vec{r_i}}$.
	
	Next we note that $R_{\vec{r_i}}$, is sent, by the map $F_1$, to coordinates on the $i$th positive axis in $\mathbb R^m$. For each $i \in \{1,...,m\}$, select $E_i$ to be an edge of $R_{\vec{r_i}}$ with exactly one vertex in $C$. Note that $\mathcal{E}=\{ [E_i] \}_{1\leq i \leq m}$. As $E_i$ is an edge of $R_{\vec{r_i}}$, we have that $E_i \subseteq H_j^-$ for all $j \neq i$. If $x \in E_i$ then $F_1(x)$=$(0,0,...,c,...,0)$ where $c$ is a positive real number in coordinate $i$. That is, $x = c \vec{e_i}$, where $\vec{e_i}$ is the $i$th standard unit coordinate vector in $\mathbb{R}^m$.
	
	We now show that by selecting $G_{s}: \mathbb{R}^m \to \mathbb{R}$ appropriately, we may induce the desired orientation on the  edges $E_i$. We may use $s$ to define $\vec{s} \in \mathbb{R}^{m}$ by $$\vec{s}=\sum_{i=1}^{m}s([E_i])\vec{e}_i.$$ Define $G_s: \mathbb{R}^m \to \mathbb{R}$ by $G_s(\vec{x})= \vec{x} \cdot \vec{s}$, without a ReLU activation, as this is the final layer of the neural network we are defining. 
	
	Define $F_s$ to be the two-layer neural network given by $G_s \circ F_1$. Note that $F_s(C)=0$ identically as $F_1(C)$ is the origin, which is preserved by $G_s$.	For $x_i \in E_i$, we note that if $F_s(x_i)>F_s(C)=0$ the induced orientation on $E_i$ is ``outwards", and if $F_s(x_i)<F_s(C)=0$ the induced orientation on $E_i$ is ``inwards". With $F_s$ defined in this way, we do indeed achieve these orientations:		
	\begin{align*}
	F_s(\vec{x}_i) &=  c(\vec{e_i} \cdot \vec{s_i}) \\ 
	&= s([E_i]) c 
	\end{align*}
	
	As $c$ was a positive real number then if $s([E_i])=-1$ then indeed $s([E_i])c<0=F_s(C)$ and if $s([E_i])=1$ then $s([E_i])c >0=F_s(C)$ as desired.
\end{proof}	

\begin{figure}
	\includegraphics[width=4in]{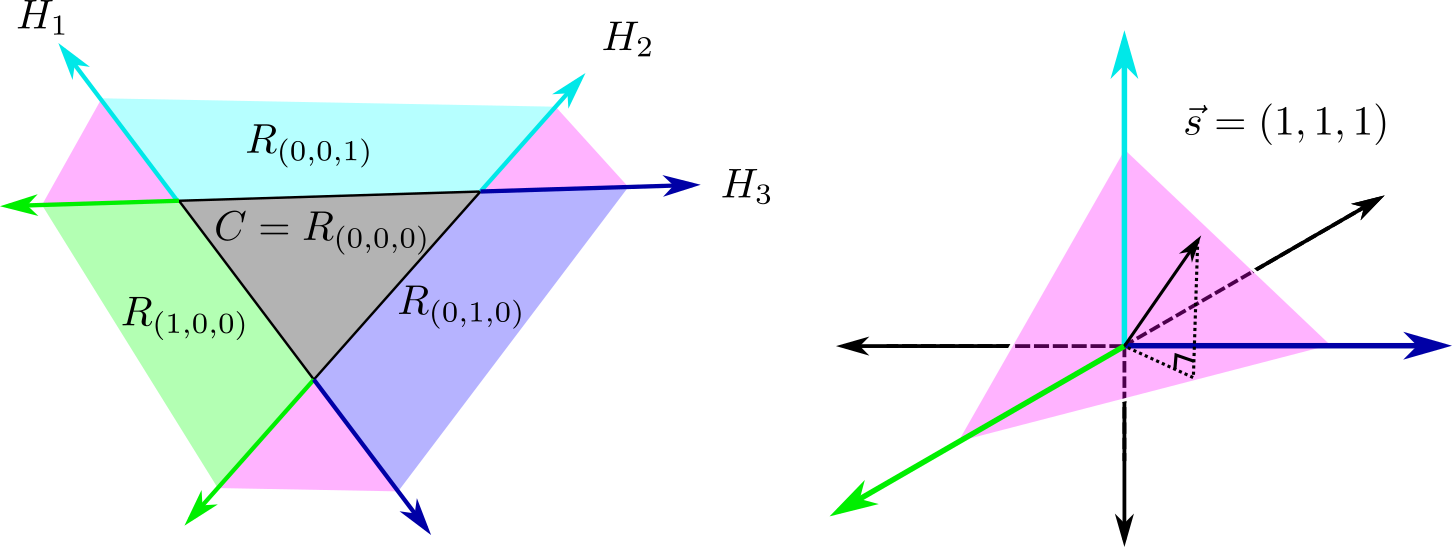}
	\caption{If $F_1$ is as pictured, then selecting $\vec{s}=(1,1,1)$ yields $G: \mathbb{R}^3 \to \mathbb{R}$ given by taking the inner product with $\vec{s}$. We may select $\vec{s}$ to be in any octant, with each octant giving a different combination of induced orientations on the pairs of colored edges. }
\end{figure}

\color{black}

	\begin{lemma}
	For each $n \in \mathbb{N}$, there is a shallow ReLU neural network $F: \mathbb R^2 \to \mathbb R^{2n+2} \to \mathbb R$, such that $\mathcal{C}(F)$ contains a flat cell $C$ with local complexity $n$. 
\end{lemma}

\begin{proof}
	Define $F_1: \mathbb R^2 \to \mathbb R^{2n+2}$ using the affine map $A_1$ with weights and biases given for $j \in \{1,\ldots,2n+2\}$: 
	
	$$\left(w_j^{(1)}\;|\;b_j^{(1)}\right) = \left( \sin\left(\frac{\pi j}{n+1} \right), \cos\left(\frac{\pi j}{n+1} \right) \Big| -1  \right) $$
	
	Then define $G: \mathbb{R}^{2n+2} \to \mathbb{R}$ with $$\left(w^{(2)}\;\big|\;b^{(2)}\right) = \left(-1, 1, -1, ... , 1, -1 |\;0\right).$$ Let $F = G \circ F_1$.
	
	We observe we may identify $H_j$ with the line tangent to the point on the unit circle at angle $\pi j/(n+1)$, co-oriented `outward' (away from the origin). Let $C$ be the flat $(2n+2)-$gon containing the origin, which satisfies $F(C)=0$. We may let $N(C)$ be a disc of radius $r$ centered at the origin, where $1<r<\left(\cos\left(\frac{\pi}{n+1}\right) \right)^{-1}$. Note that $N(C)$ contains no other vertices of $\mathcal{C}(F)$ except those of $C$.\\
	
	Now consider $R_j$, the region given by $$R_j = N(C) \cap H_j^+ \cap \left(\bigcap_{i \neq j}H_i ^-\right)$$ 
	
	If $j$ is even, $F(R_j)>0$. Likewise, if $j$ is odd, $F(R_j)<0$. As a result, for small $\epsilon>0$, $F^{-1}(-\infty, -\epsilon]\cap N(C)$ has $n+1$ connected components whereas $F^{-1}(-\infty, \epsilon]\cap N(C)$ is connected. Since attaching a PL handle may remove at most one connected component, at least $n$ handles must be attached to $ F^{-1}(-\infty, -\epsilon]\cap N(C)$ to obtain a PL manifold homeomorphic to $F^{-1}(-\infty, \epsilon]\cap N(C)$. Thus, the local complexity of $C$ is at least $n$. 
	
	Finally, it is sufficient to add $n$ handles, as pictured below in the case when $n=2$. 	
	
\end{proof}
\begin{figure}[h]
	\begin{minipage}{0.4\textwidth}
		\includegraphics[width=2in]{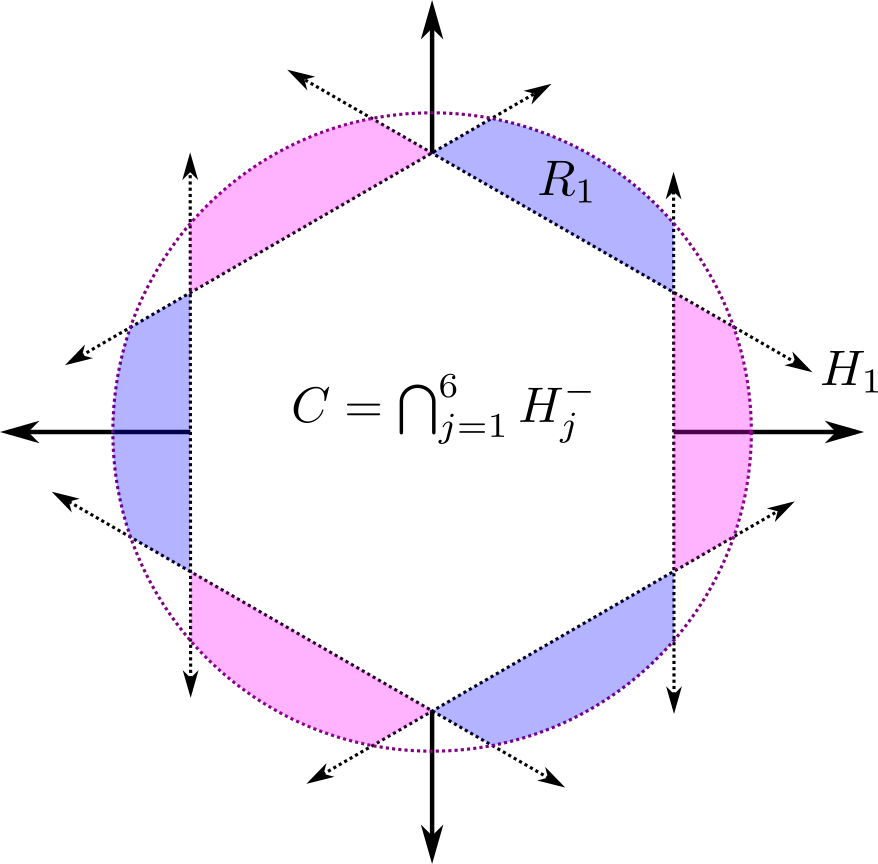}
	\end{minipage}	\begin{minipage}{0.4\textwidth}
		\includegraphics[width=2in]{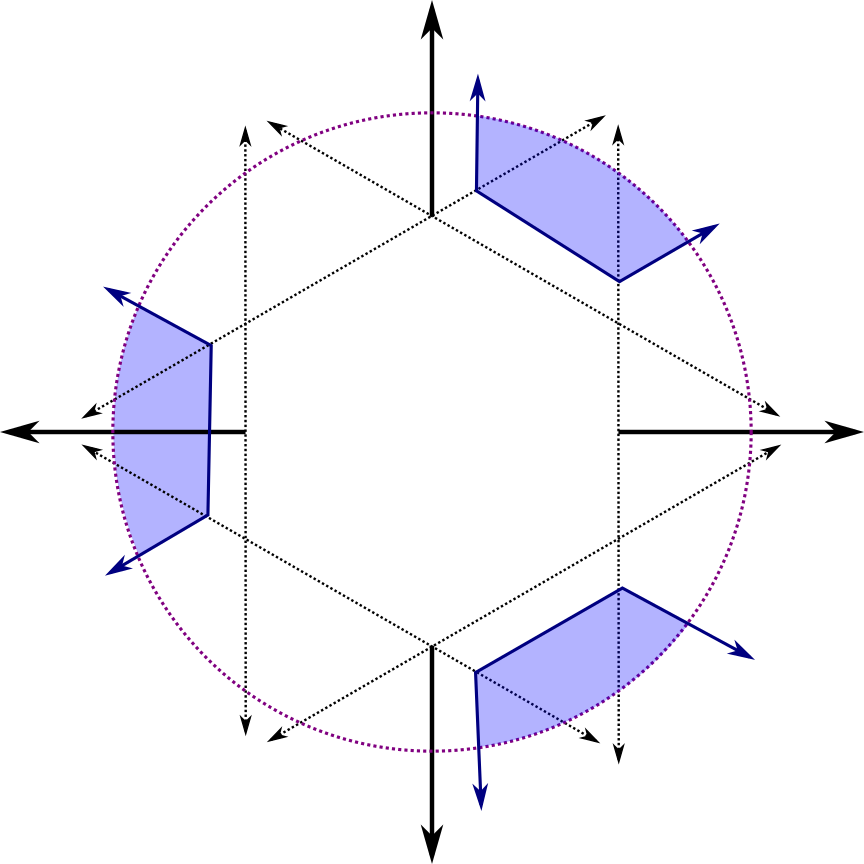}
	\end{minipage} 
	
	\begin{minipage}{0.4\textwidth}
		\includegraphics[width=2in]{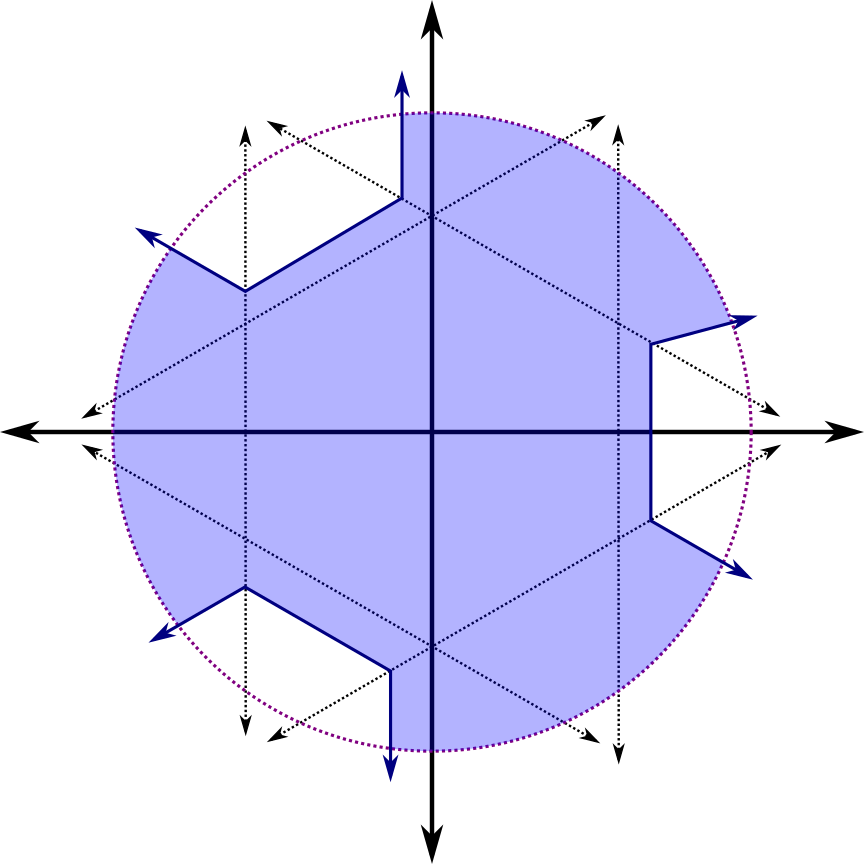}
	\end{minipage}
	\begin{minipage}{0.4\textwidth}
		\includegraphics[width=2in]{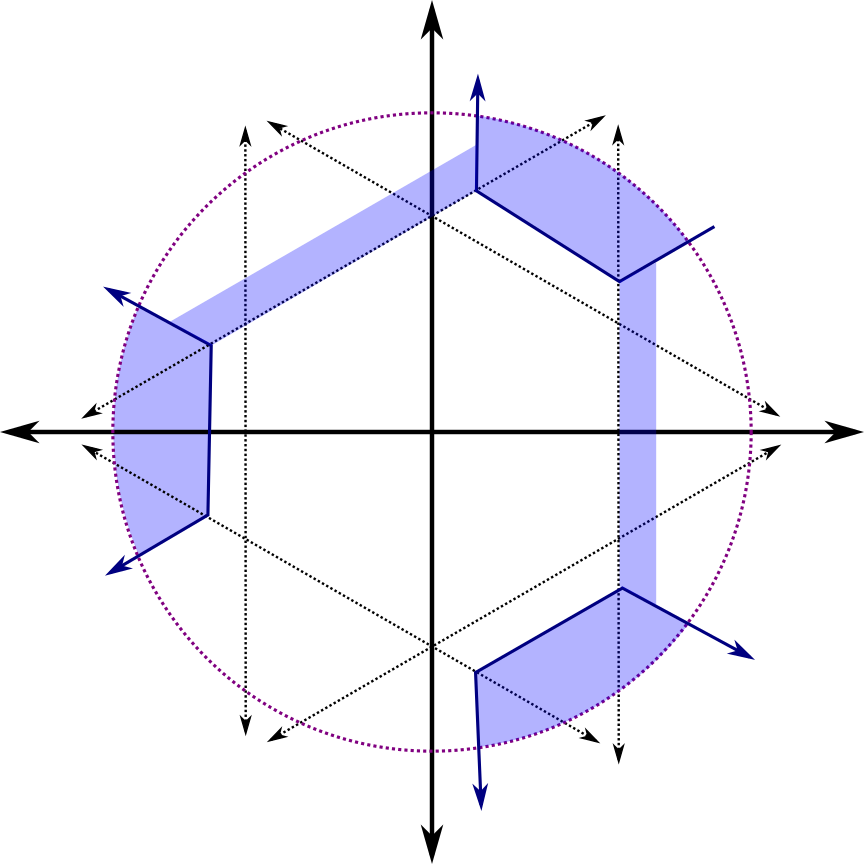}
	\end{minipage}
	\caption{Example when $n=2$. \textit{Top Left:} the polyhedral complex $\mathcal{C}(F)$ with $C$, $R_1$ and $H_1$ distinguished. The circle bounds $N(C)$. $F$ is strictly positive in the pink cells, and strictly negative in the blue cells. \textit{Top Right:} The sublevel set $F^{-1}(-\infty, -\epsilon]$ is shaded. \textit{Bottom Left:} The sublevel set $F^{-1}(-\infty, \epsilon]$ is shaded. \textit{Bottom Right:} Two handles are attached to $F^{-1}(-\infty, -\epsilon]\cap N(C)$ yielding a polyhedral complex which is PL homeomorphic to  $F^{-1}(-\infty, -\epsilon] \cap N(C)$.  }
\end{figure}

\begin{lemma}
	The closure $K$ of the flat cell $C$ given in the previous lemma has local $H_1$-complexity $n$, and its local $H_i$-complexity is zero for all $i \neq 1$.
\end{lemma}

\begin{proof}
	We compute $H_*(F_{\leq 0}, F_{\leq 0}\setminus K)$
We note by excision (cf. \cite{Hatcher}) that we may instead compute
	$$H_*(F_{\leq 0} \cap N(C), (F_{\leq 0} \setminus C) \cap N(C))  $$ 
Since  $(F_{\leq 0}\setminus K) \cap N(C) $ is homotopy equivalent to $n+1$ points and $F_{\leq 0} \cap N(C)$ is homotopy equivalent to a point, applying the long exact sequence for reduced homology gives us trivial relative homology in all degrees except degree 1, where it induces an isomorphism 
	$$ H_1(F_{\leq 0}\cap N(C),    (F_{\leq 0} \setminus K)\cap N(C)) \cong \tilde{H}_0((F_{\leq 0} \setminus K) \cap N(C))$$ 

	Thus, ${H}_1(F_{\leq 0}, F_{\leq 0} \setminus K) \cong \mathbb{Z}^{n}$.
	
\end{proof}

\color{black}
\bibliography{ReLUPaper}

\def\polhk#1{\setbox0=\hbox{#1}{\ooalign{\hidewidth
  \lower1.5ex\hbox{`}\hidewidth\crcr\unhbox0}}}
\begin{thebibliography}{10}

\bibitem{AroraBasu}
Raman Arora, Amitabh Basu, Poorya Mianjy, and Anirbit Mukherjee.
\newblock Understanding deep neural networks with rectified linear units.
\newblock In {\em 6th International Conference on Learning Representations,
  {ICLR} 2018, Vancouver, BC, Canada, April 30 - May 3, 2018, Conference Track
  Proceedings}. OpenReview.net, 2018.

\bibitem{Banchoff}
Thomas Banchoff.
\newblock Critical points and curvature for embedded polyhedra.
\newblock {\em J. Differential Geometry}, 1:245--256, 1967.

\bibitem{BianchiniScarselli}
Monica Bianchini and Franco Scarselli.
\newblock On the complexity of neural network classifiers: {A} comparison
  between shallow and deep architectures.
\newblock {\em {IEEE} Trans. Neural Networks Learn. Syst.}, 25(8):1553--1565,
  2014.

\bibitem{GL}
J.~Elisenda Grigsby and Kathryn Lindsey.
\newblock On transversality of bent hyperplane arrangements and the topological
  expressiveness of {R}e{LU} neural networks.
\newblock Preprint: http://arxiv.org/abs/2008.09052, To appear: {\em Siam
  Journal on Applied Algebra and Geometry}, 2022.

\bibitem{Grunbaum}
Branko Gr\"{u}nbaum.
\newblock {\em Convex polytopes}, volume 221 of {\em Graduate Texts in
  Mathematics}.
\newblock Springer-Verlag, New York, second edition, 2003.
\newblock Prepared and with a preface by Volker Kaibel, Victor Klee and
  G\"{u}nter M. Ziegler.

\bibitem{Grunert}
Romain Grunert.
\newblock {\em Piecewise Linear Morse Theory}.
\newblock PhD thesis, Freie Universit{\"a}t Berlin, 2016.
\newblock https://refubium.fu-berlin.de/handle/fub188/12531.

\bibitem{GrunertRote}
Romain Grunert, Wolfgang K{\"u}hnel, and G{\"u}nter Rote.
\newblock {PL} {M}orse theory in low dimensions.
\newblock Preprint: https://arxiv.org/pdf/1912.05054.pdf, 2019.

\bibitem{GussSalakhutdinov}
William~H. Guss and Ruslan Salakhutdinov.
\newblock On characterizing the capacity of neural networks using algebraic
  topology.
\newblock {\em CoRR}, abs/1802.04443, 2018.

\bibitem{HaninRolnick}
Boris Hanin and David Rolnick.
\newblock Deep {R}e{LU} networks have surprisingly few activation patterns.
\newblock {\em CoRR}, abs/1906.00904, 2019.

\bibitem{Hatcher}
Allen Hatcher.
\newblock {\em Algebraic topology}.
\newblock Cambridge University Press, Cambridge, 2002.

\bibitem{RS}
C.~P. Rourke and B.~J. Sanderson.
\newblock {\em Introduction to piecewise-linear topology}.
\newblock Ergebnisse der Mathematik und ihrer Grenzgebiete, Band 69.
  Springer-Verlag, New York-Heidelberg, 1972.

\bibitem{Schrijver}
Alexander Schrijver.
\newblock {\em Theory of linear and integer programming}.
\newblock Wiley-Interscience Series in Discrete Mathematics. John Wiley \&
  Sons, Ltd., Chichester, 1986.
\newblock A Wiley-Interscience Publication.

\bibitem{Stanley}
Richard~P. Stanley.
\newblock An introduction to hyperplane arrangements.
\newblock In {\em Geometric combinatorics}, volume~13 of {\em IAS/Park City
  Math. Ser.}, pages 389--496. Amer. Math. Soc., Providence, RI, 2007.

\bibitem{Ziegler}
G\"{u}nter~M. Ziegler.
\newblock {\em Lectures on polytopes}, volume 152 of {\em Graduate Texts in
  Mathematics}.
\newblock Springer-Verlag, New York, 1995.

\bibitem{ZomorodianCarlsson}
Afra Zomorodian and Gunnar Carlsson.
\newblock Computing persistent homology.
\newblock {\em Discrete Comput. Geom.}, 33(2):249--274, 2005.

\end{thebibliography}
\end{document}